\documentclass[graybox]{svmult}

\usepackage{type1cm}        
\usepackage{makeidx}         
\usepackage{graphicx}        
\usepackage{multicol}        
\usepackage[bottom]{footmisc}

\usepackage{newtxtext}       %
\usepackage[varvw]{newtxmath}       

\usepackage{tikz-cd}
\usetikzlibrary{decorations.pathmorphing}

\usepackage{latexsym}
\usepackage{color}
\usepackage{enumitem}
\usepackage{mathtools}
\usepackage{amsrefs}

\newcommand{\norm}[1]{\lVert#1\rVert}

\newcommand{\maps}[1]{\mathop{\overset{#1}\longrightarrow}}
\newcommand {\Del}{ \; \Big| \;}
\newcommand {\del}{ \; \big| \;}

\newcommand{\iso}{\cong}
\newcommand{\ov}{\overline}
\newcommand{\inv}{^{-1}}

\newcommand {\N}{\mathbb N}
\newcommand {\Z}{\mathbb Z}

\newcommand {\R}{\mathbb R}
\newcommand {\OO}{\mathcal O}
\newcommand{\eps}{\epsilon}
\newcommand{\equi}{\Leftrightarrow}

\newcommand {\go} {\mathfrak}
\newcommand{\saa}{\Rightarrow}
\newcommand {\ku} {\mathcal}
\newcommand {\e} {\exists}
\renewcommand {\a} {\forall}

\theoremstyle{plain}
\newtheorem{exa}{Example}
\newtheorem{observation}{Observation}

\definecolor{roed}{rgb}{1,0,0}

\newcommand{\forkindep}[1][]{\mathop{\mathop{\vcenter{\hbox{\oalign{\noalign{\kern-.3ex}
\hfil$\vert$\hfil\cr\noalign{\kern-.7ex}$\smile$\cr\noalign{\kern-.3ex}}}}}\displaylimits_{#1}}}

\newcommand{\maths}[1]{\[\begin{split}{#1}\end{split}\]}

\newcommand{\mathseq}[2]{\begin{equation}\label{#1}\begin{split}{#2}\end{split}\end{equation}}

\makeindex             

\begin{document}

\title*{The geometrisation problem for topological groups}
\author{Christian Rosendal}
\institute{Christian Rosendal \at Department of Mathematics, University of Maryland, 4176 Campus Drive - William E. Kirwan Hall, College Park, MD 20742-4015, USA\\
{rosendal@umd.edu}\\
{https://sites.google.com/view/christian-rosendal/}
}
%
%
\maketitle

\abstract*{Each chapter should be preceded by an abstract (no more than 200 words) that summarizes the content. The abstract will appear \textit{online} at \url{www.SpringerLink.com} and be available with unrestricted access. This allows unregistered users to read the abstract as a teaser for the complete chapter.
Please use the 'starred' version of the \texttt{abstract} command for typesetting the text of the online abstracts (cf. source file of this chapter template \texttt{abstract}) and include them with the source files of your manuscript. Use the plain \texttt{abstract} command if the abstract is also to appear in the printed version of the book.}

\section{Introduction}
In this paper we  present a general approach to what may be termed the {\em geometrisation problem} for topological groups, namely,  the problem of understanding to which extent and how topological groups are also geometric objects. Of course, as is familiar from the case of Lie groups, topological groups often come with some pre-determined geometry resulting, for example, from a differentiable structure on the group. However, our goal differs from Lie theory in the sense that we are not aiming to describe the geometry obtained by imposing some additional structure on a topological group, but rather to investigate how geometric structure arises from the topological-algebraic structure of the group itself. Thus, the geometries we find should be inherent to the topological group structure. This viewpoint is of course familiar to geometric group theorists. Indeed, the basic premise of geometric group theory is the observation that the word metric on a finitely generated group is independent (up to bi-Lipschitz equivalence) of the chosen finite generating set. In that sense, a finitely generated group can be said to have a well-defined Lipschitz geometry.

The specific viewpoint taken here is to split this geometrisation problem into two questions, one dealing with the small scale and the other with the large scale. In turn, this is done by using Lipschitz conditions on maps to introduce a couple of geometric categories, 
$$
{\sf LocalLipschitz}, 
$$
reflecting possible small scale geometry, and 
$$
{\sf Quasimetric},
$$
which codifies possible large scale geometry. These categories are themselves more structured versions of the categories of uniform and coarse spaces to which all topological groups naturally belong.  However, to describe all of this, we need first to introduce the fundamental concepts.

\begin{acknowledgement}
This article is based on two mini-courses given by the author at the workshop {\em Big Mapping Class Groups and Diffeomorphism Groups} at CIRM, Marseille, France in October 2022 and at the Korea Institute for Advanced Studies (KIAS), Seoul, South Korea in June 2023. The author is grateful to the participants and organisers of each of these events and especially to Prof. Sanghyun Kim  for the opportunity to visit KIAS and enjoy the fruitful work environment there. Many thanks are also due to the anonymous referee for a very detailed report that helped to improve the exposition.

Much of the material presented here has found its way into other publications. In particular, there is overlap in topics with the book \cite{coarsebook} and the recently published paper \cite{bull}.  For the new material presented in later sections, complete proofs are provided.
\end{acknowledgement}


\section{Lipschitz maps}
As noted above, the initial geometric categories we will consider can be  arrived at by considering various Lipschitz type conditions that can hold of maps between metric spaces. 
\begin{definition}
A map $X\overset\phi\longrightarrow M$ between metric spaces $(X,d)$ and $(M,\partial)$ is 
\begin{itemize}
\item {\em Lipschitz} if there is a constant $K$ such that, for all $x,y\in X$, 
$$
\partial(\phi x, \phi y)\leqslant K\cdot d(x,y),
$$
\item {\em Lipschitz for large distances} if there is a constant $K$ such that, for all $x,y\in X$, 
$$
\partial(\phi x, \phi y)\leqslant K\cdot d(x,y) +K,
$$
\item {\em Lipschitz for short distances} if there are constants $K$ and $\delta>0$ such that, 
$$
\partial(\phi x, \phi y)\leqslant K\cdot d(x,y)
$$
whenever $x,y\in X$ satisfy $d(x,y)\leqslant \delta$.
\end{itemize}
\end{definition}
It should be obvious from the definition above that each of these classes of maps is closed under composition. For example, if 
$X\overset\phi\longrightarrow Y\overset\psi\longrightarrow Z$ are maps between metric spaces that are both Lipschitz for short distances, then so is the composition $\psi\circ \phi$. 

Equally immediate, but perhaps less expected, is the fact that these concepts allow us a splitting of the Lipschitz condition as a conjunction of two weaker properties.

\begin{lemma}\label{lipschitz}
For any map $X\overset\phi\longrightarrow M$ between metric spaces, we have 
\maths{
\phi \text{ is Lipschitz }\;\equi\; \phi \text{ is Lipschitz for both large and short distances.}
}
\end{lemma}

\begin{proof}
Indeed, suppose that $\phi$ is both Lipschitz for large and short distances and pick constants $K$ and $\delta>0$ such that, for all $x,y\in X$, 
$$
\partial(\phi x, \phi y)\leqslant K\cdot d(x,y) +K
$$
and so that 
$$
\partial(\phi x, \phi y)\leqslant K\cdot d(x,y)
$$
whenever $x,y\in X$ satisfy $d(x,y)\leqslant \delta$.
Then, given $x,y\in X$, either $d(x,y)\leqslant \delta$, in which case
$$
\partial(\phi x, \phi y)\leqslant K\cdot d(x,y),
$$
or $d(x,y)> \delta$, in which case
$$
\partial(\phi x, \phi y)\leqslant K\cdot d(x,y) +K < (K+\tfrac K\delta)\cdot d(x,y).
$$
Thus, $\phi$ is Lipschitz with constant $K+\tfrac K\delta$.
\end{proof}

As the different classes maps are closed under composition, they induce equivalence relations on the family of metrics on any single space $X$.
\begin{definition}
We define the relations of {\em bi-Lipschitz}, {\em quasi\--iso\-metric} or {\em locally bi-Lipschitz} equivalence on the collection of all metrics on a set $X$ by
\begin{align*}
d\sim_{\sf Lip} \partial 
&\equi \textrm{$\big(X,d\big)\underset{\sf id}{\overset{\sf id}\rightleftarrows} \big(X,\partial\big)$ are both Lipschitz}\\
&\equi \e K\;\; \frac 1K d\leqslant \partial \leqslant K \cdot d,\\
d\sim_{\sf QI} \partial 
&\equi \textrm{$\big(X,d\big)\underset{\sf id}{\overset{\sf id}\rightleftarrows} \big(X,\partial\big)$ are both Lipschitz for large distances}\\
&\equi \e K\;\; \frac 1K d- K\leqslant \partial \leqslant K \cdot d +K,\\
d\sim_{\sf locLip} \partial 
&\equi \textrm{$\big(X,d\big)\underset{\sf id}{\overset{\sf id}\rightleftarrows} \big(X,\partial\big)$ are both Lipschitz for short distances}.
\end{align*}
\end{definition}

\begin{exa}
The standard euclidean metric $d_1(x,y)=|x-y|$ on $\R$ is locally Lipschitz equivalent with the truncated metric $d_2(x,y)=\min\{1,|x-y|\}$.
On the other hand, $d_1$ is unbounded, whereas $d_2$ is bounded, so $d_1$ and $d_2$ are not quasi-isometric.
 
Similarly, the euclidean metric $d_1$ is quasi-isometric with 
$d_3(x,y)=\big\lceil |x-y|\big\rceil$, where $\lceil t\rceil$ denotes the upper integral part of a number $t$, but $d_1$ and $d_3$ are not locally Lipschitz equivalent. 

And finally, because the map $t\mapsto \sqrt t$ is not Lipschitz for short distances, whereas $t\mapsto t^2$ is not Lipschitz for large distances, $d_1$  is neither quasi-isometric nor   locally Lipschitz equivalent with the metric
$$
d_4(x,y)=\sqrt{|x-y|}.
$$
\end{exa}

The various concepts introduced above allow us to construct three categories
$$
{\sf Lipschitz}, \hspace{1cm} {\sf LocalLipschitz},\hspace{1cm}{\sf Quasimetric}
$$
whose objects are arbitrary sets equipped with respectively a bi-Lipschitz equivalence class of metrics, a locally bi-Lipschitz equivalence class of metrics or a quasi-isometry class of metrics.

By design, a function $X\overset \phi\longrightarrow M$ being Lipschitz, respectively Lipschitz for short or large distances, does not depend on the specific choice of metrics on $X$ and $M$, but only the corresponding equivalence classes of the metrics. Consequently, we may define the class of morphisms in the first two categories as the maps $X\overset \phi\longrightarrow M$ that are respectively Lipschitz or Lipschitz for short distances with respect to some or, equivalently, any choice of metrics from the selected equivalence classes. 

On the other hand, because ultimately we will want to count objects of different cardinalities as isomorphic in the quasimetric category, morphisms cannot be maps in the set theoretical sense. Instead, let us say that two maps $X\overset{\phi, \psi}\longrightarrow M$ between metric spaces are  {\em close} if 
$$
\sup_{x\in X}d\big(\phi x, \psi x\big)<\infty
$$
and note that, in this case, either both $\phi$ and $\psi$  are Lipschitz for large distances or none of them are. A morphism $X\overset {\Phi}\longrightarrow M$ in the category {\sf Quasimetric} is then simply a closeness class of functions from $X\maps\phi M$ that are  Lipschitz for large distances. In particular, this has the consequence that two quasimetric spaces $X$ and $M$  are isomorphic as quasimetric spaces if and only if there are two functions
$$
X\overset \phi \longrightarrow M \overset \psi \longrightarrow X,
$$
both Lipschitz for large distances, so that the compositions $\psi\circ \phi$ and $\phi\circ \psi$ are close to the identities on $X$ and $M$ respectively. 

\begin{exa}
The prototypical example of quasi-isometric spaces is $X=\R$ and $M=\Z$, both equipped with the quasi-isometric equivalence class of the euclidean distance, and where $\phi=\lfloor\cdot\rfloor$ is the lower integral part, whereas $\psi={\sf id}$. 
\end{exa}
Isomorphisms in the category {\sf Quasimetric} are termed {\em quasi-isometries}. In practice however, we shall  not distinguish between morphisms $\Phi$ and specific function representatives $\phi\in \Phi$ of these.

For completeness, let us also let 
$$
{\sf Metric}
$$
denote the category of metric spaces in which morphisms are (not necessarily surjective) isometries.

Evidently, every metric space  gives rise to an object in the category {\sf Lipschitz} by  simply forgetting the specific metric and only remembering its bi-Lipschitz equivalence class. Similarly, coarsening the bi-Lipschitz class to  the local bi-Lipshitz or quasimetric class, we pass to the categories  {\sf LocalLipschitz} and {\sf Quasimetric}, respectively. In this manner, we obtain a diagram, Figure \ref{first diagram}, of categories connected by {\em forgetful functors} that to each object of the domain category associates a less structured {\em reduct} in the target category.

\begin{figure}
\begin{tikzcd}
&&&&&{\sf Metric} \arrow[]{d}&\\
&&&& &{\sf Lipschitz} \arrow[]{dl}{}  \arrow[]{dr}{} &\\
&&&&{\sf Local Lipschitz} &  & {\sf Quasimetric} \\
\end{tikzcd}
\caption{Forgetful functors between geometric categories}
\label{first diagram}
\end{figure}
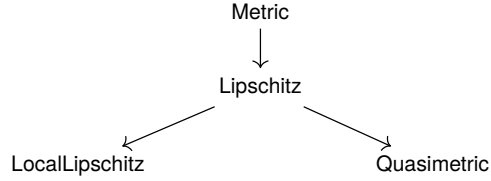


\section{Uniform and coarse spaces}
Evidently, every map between metric spaces that is Lipschitz for short distances is automatically uniformly continuous. In particular, this means that the uniform structures $\ku U_d$ and $\ku U_\partial$ given by two locally Lipschitz equivalent metrics $d$ and $\partial$ must coincide, i.e., $\ku U_d=\ku U_\partial$. However, to give a proper presentation of this and also to motivate the category of coarse spaces, recall the definition of uniform structures due to A. Weil \cite{weil}. 

\begin{definition}
A {\em uniform structure} on a set $X$  is a {filter} $\ku U$ of subsets $E\subseteq X\times X$, called {\em entourages}, satisfying
\begin{enumerate}
\item $\Delta \subseteq E$ for all $E\in \ku U$,
\item  if $E\in \ku U$, then $E^{\sf T}=\{(y,x)\del (x,y)\in E\}\in \ku U$,
\item if $E\in \ku U$, then $F\circ F=\{(x,z)\del \e y\, (x,y),(y,z)\in F\}\subseteq E$ for some $F\in \ku U$.
\end{enumerate}
A {\em uniform space} is just a set equipped with a uniform structure.
\end{definition}
Here $\Delta=\{(x,x)\del x\in X\}$ denotes the diagonal in $X\times X$. Recall also that a {\em filter} on a set $Z$ is a family $\ku F$ of subsets of $Z$ containing $Z$ and closed under supersets and finite intersections. 

\begin{exa}
If $d$ is a pseudometric (or \'ecart) on a set $X$, i.e., if $d$ is a metric except that possibly $d(x,y)=0$ for distinct $x,y\in X$, then the induced uniform structure $\ku U_d$ is the filter generated by the  family of entourages
$$
E_\alpha=\{(x,y)\del d(x,y)<\alpha\}
$$
for $\alpha>0$. That is, for a set $E\subseteq X\times X$, we have 
\begin{equation}\label{uniformpseudo}
E\in \ku U_d\quad\equi\quad \inf\{d(x,y)\del (x,y)\notin E\}>0.
\end{equation}
\end{exa}

Also, a morphism between two uniform spaces $(X,\ku U)$ and $(M,\ku V)$  is simply a uniformly continuous map $X\overset{\phi}{\longrightarrow} M$, meaning that, for every entourage $E\in \ku V$ (corresponding to $\eps>0$), there is an entourage $D\in \ku U$ (corresponding to $\delta>0$) such that 
$$
(x,y)\in D \saa (\phi x, \phi y)\in E.
$$
We let 
$$
{\sf Uniform}
$$ 
denote the category of uniform spaces. As noted above, if two metrics $d$ and $\partial$ are locally Lipschitz equivalent, then their uniform structures coincide. So this gives a forgetful functor 
$$
{\sf LocalLipshitz}\to {\sf Uniform}
$$ 

In analogy with the definition of uniform space, J. Roe \cite{roe} codified a notion of coarse space, which is related  to quasimetric spaces in the same way that uniform spaces are related to local Lipschitz spaces.

\begin{definition}
A {\em coarse structure} $\ku E$ on a set  $X$ is an ideal of subsets  $E\subseteq X\times X$, again called {\em entourages}, so that the diagonal  $\Delta$ belongs to $\ku E$ and 
\begin{itemize}
\item $E^{\sf T}\in \ku E$ and $E\circ E\in \ku E$  whenever $E\in \ku E$.
\end{itemize}
A {\em coarse space} is just a set equipped with a coarse structure.
\end{definition}

\begin{exa}
The coarse structure $\ku E_d$ of a pseudometric space $(X,d)$ is the ideal (rather than filter as in the case of the uniform structure) generated by the entourages 
$$
E_\alpha=\{(x,y)\del d(x,y)<\alpha\}
$$
for $\alpha<\infty$. That is, for a set $E\subseteq X\times X$, we have 
\begin{equation}\label{coarsepseudo}
E\in \ku E_d\quad\equi\quad \sup\{d(x,y)\del (x,y)\in E\}<\infty.
\end{equation}
\end{exa}

For a coarse structure, the focus is thus on the asymptotic behaviour of entourages $E_\alpha$ as $\alpha\to \infty$, as opposed to the uniform structure where one is primarily interested in the behaviour as $\alpha\to 0$. 

\begin{observation}As can be seen from (\ref{uniformpseudo}) and (\ref{coarsepseudo}), at least in the case of pseudometric spaces, there is a sort of duality between the uniform and coarse structures. Namely, for a function $d\colon X\times X\to [0,\infty]$, define two families
\maths{
\ku U_d&=\big\{E\subseteq X\times X\del  \inf\{d(x,y)\del (x,y)\notin E\}>0\big\},\\
\ku E_d&=\big\{E\subseteq X\times X\del  \sup\{d(x,y)\del (x,y)\in E\}<\infty\big\}.
}
Then, for all $E\subseteq X\times X$, we have 
$$
\complement E\in \ku E_d\quad\equi  \quad E\in \ku U_{\tfrac 1d},
$$
where $\complement E=X\times X\setminus E$ denotes the complement of $E$ in $X\times X$.
Therefore, when $d$ is a pseudometric on $X$, the coarse structure $\ku E_d$ is just the dual\footnote{The {\em dual} of a family $\ku F\subseteq \ku P(\Omega)$ is the collection of complements of elements of $\ku F$, that is, the family $\ku F^*=\{\complement A\del A\in \ku F\}$. This constitutes a duality between ideals and filters on a set $\Omega$.} ideal of $\ku U_{\tfrac 1d}$, whereas $\ku U_d$ is the dual filter of $\ku E_{\tfrac 1d}$.
\end{observation}

Any two quasi-isometric metrics on a set $X$ will induce the same coarse structure and therefore every quasi-metric space defines a unique coarse structure on the same underlying set. To make the class of coarse spaces into a category {\sf Coarse}, we just need to describe the morphisms. For this, a map $X\overset\phi\longrightarrow M$ between coarse spaces $(X,\ku E)$ and $(M,\ku F)$ is said to be {\em controlled} if, for every entourage $E\in \ku E$, there is an entourage $F\in \ku F$ so that 
$$
(x,y)\in E\saa (\phi x,\phi y)\in F.
$$
When $\ku E$ and $\ku F$ are the coarse structures associated with metrics $d$ on $X$ and  $\partial$ on $M$, this just means that there is a monotone increasing function $\omega\colon \R_{\geqslant 0}\to \R_{\geqslant 0}$ for which $$\partial  (\phi x,\phi y)\leqslant \omega\big(d(x,y)\big)$$ for all $x,y\in X$.

\begin{figure}
\begin{tikzcd}
&&&&&{\sf Metric} \arrow[]{d}&\\
&&&& &{\sf Lipschitz} \arrow[]{dl}{}  \arrow[]{dr}{} &\\
&&&&{\sf Local Lipschitz} \arrow[]{d}{} &  & {\sf Quasimetric}\arrow[]{d}{} \\
&&&&{\sf Uniform} & &{\sf Coarse}
\end{tikzcd}
\caption{Expanded diagram of forgetful functors between  geometric categories}
\label{second diagram}
\end{figure}
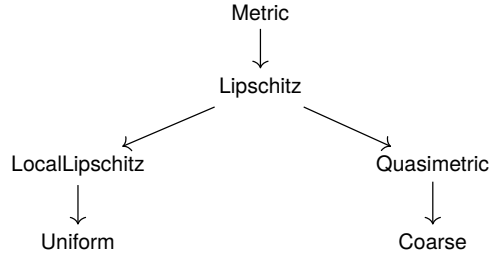

Morphisms in  {\sf Coarse} are then closeness classes of controlled maps, where  two maps $X\overset{\phi,\psi}\longrightarrow M$ are {\em close} provided that there is an entourage $F$ on $M$ so that $(\phi x, \psi x)\in F$ for all $x\in X$.
With this definition, the  forgetful map 
$$
{\sf Quasimetric}\longrightarrow {\sf Coarse}
$$ 
becomes a functor.


\section{The uniform and coarse structure on a topological group}
On every topological group $G$, one can define several uniform structures of interest. In this connection, the left uniform structure is the most appropriate.
\begin{definition}
The {\em left uniform structure}  on a topological group $G$ is the ideal $\ku U_L$ on $G\times G$ generated by the family of entourages
$$
E_V=\{(g,f)\in G\times G\del g\inv f\in V\},
$$
where $V$ ranges over identity neighbourhoods in $G$. 
\end{definition}
We note that the left uniform structure is {\em compatible} with the topology on $G$ in the sense that, for every $g\in G$, sets of the form
$$
E[g]=\{f\in G\del (f,g)\in E\},
$$
with $E\in \ku U_L$, form a neighbourhood basis at $g$.

As always, metrisable uniform structures are generally much easier to manipulate and hence the following result is fundamental. To avoid unnecessary complications, henceforth, all topological groups will be assumed to be Hausdorff. 

\begin{theorem}\cites{birkhoff,kakutani}\label{birkhoff-kakutani}{}
The following conditions are equivalent for any  topological group $G$.
\begin{enumerate}
\item The topology on $G$ is first countable,
\item the topology on $G$ is metrisable, 
\item the left uniform structure on $G$  is metrisable.
\end{enumerate}
Furthermore, in this case, $G$ admits a {\em compatible, left-invariant} metric $d$, i.e., inducing the topology on $G$ and so that
$$
d(hg,hf)=d(g,f)
$$ 
for all $f,g,h\in G$.
Moreover,  every such compatible left-invariant metric  will induce the left-uniform structure on $G$.
\end{theorem}

Underlying this is a basic metrisation technique due to G. Birkhoff \cite{birkhoff} that is useful in many other contexts. For transparency, we formulate it for the associated pseudolength function. So let us recall that a {\em pseudolength function} on a group $G$ is a map $\ell\colon G\to \R_{\geqslant 0}$ so that, for all $x,y\in G$, 
\begin{enumerate}
\item $\ell(1)=0$,
\item $\ell(xy)\leqslant \ell(x)+\ell(y)$,
\item $\ell(x\inv)=\ell(x)$.
\end{enumerate}
Observe also that there is a bijective correspondence between pseudolength functions $\ell$ and left-invariant pseudometrics $d$ on $G$ given by $\ell(x)=d(x,1)$ and $d(x,y)=\ell(x\inv y)$. Moreover, the pseudolength $\ell$ is an actual {\em length function}, i.e., $\ell(x)\neq 0$ for all $x\neq 1$, if and only if the corresponding pseudometric $d$ is a  metric. Let us also recall that a subset $A$ of a group $G$ is said to be {\em symmetric} in case
$$
A\inv:= \{a\inv\del a\in A\}=A.
$$

\begin{lemma}\label{lemma:birkhoff}
Let $G$ be a topological group and $(V_n)_{n\in\Z}$
an increasing chain of symmetric open identity neighbourhoods satisfying  $G=\bigcup\limits_{n\in \Z}V_n$ and  $V_n^3\subseteq V_{n+1}$ for all $n\in \Z$.
Define
$$
L(g)=\inf\big\{2^n\del g\in V_n\big\}
$$ 
and put
$$
\ell(g)=\inf \Big\{\sum_{i=1}^{k}L(h_i)\Del g=h_1h_2\cdots h_k\Big\}.
$$
Then
$$
\tfrac 12L(g)\leqslant  \ell(g)\leqslant L(g)
$$
and $\ell$ is a continuous pseudolength function on $G$.

Moreover, if $(V_n)_{n\in\Z}$ is actually a neighbourhood basis at the identity, then $\ell$ is a length function and the associated metric $d$ induces the topology on $G$.
\end{lemma}

\begin{proof}
Let us just verify the bounds $\frac 12L(g)\leqslant  \ell(g)\leqslant L(g)$ of which only the first inequality is non-trivial.  By induction on $k\geqslant 1$ we show that, for all products $h_1\cdots h_k$, one has
\begin{equation}\label{bk}
L(h_1\cdots h_k)\leqslant 2\cdot \sum_{i=1}^{k}L(h_i).
\end{equation}
This will suffice to show that $\frac 12L(g)\leqslant  \ell(g)$ for all $g\in G$. 

So let $k$ be given and suppose that (\ref{bk}) holds for all smaller values of $k$. Consider a product $h_1\cdots h_k$. To see that (\ref{bk}) holds, note that, if $\sum_{i=1}^{k}L(h_i)=0$, then $h_i\in \bigcap_{n\in \Z}V_n$ for all $i$ and so also $h_1\cdots h_k\in \bigcap_{n\in \Z}V_n$, whereby also $L(h_1\cdots h_k)=0$. So assume instead that $\sum_{i=1}^{k}L(h_i)>0$ and let $0\leqslant n< k$ be maximal so that 
\begin{equation}\label{1}
\sum_{i=1}^{n}L(h_i)\leqslant \frac 12\sum_{i=1}^{k}L(h_i), 
\end{equation}
whereby
$$
\sum_{i=1}^{n+1}L(h_i)> \frac 12\sum_{i=1}^{k}L(h_i)
$$
and so 
\begin{equation}\label{2}
\sum_{i=n+2}^{k}L(h_i)< \frac 12\sum_{i=1}^{k}L(h_i).
\end{equation}
It follows by the inductive hypothesis and inequalities  (\ref{1}) and (\ref{2}) that
$$
2^m:=\max\{L(h_1\cdots h_n), L(h_{n+1}),    L(h_{n+2}\cdots h_k)\}\leqslant \sum_{i=1}^{k}L(h_i),
$$
whereby 
$$
h_1\cdots h_k=(h_1\cdots h_n)\cdot h_{n+1}\cdot (h_{n+2}\cdots h_k)\in V_m\cdot V_m\cdot V_m\subseteq V_{m+1}
$$
and hence
$$
L(h_1\cdots h_k)\leqslant 2^{m+1}\leqslant 2\cdot \sum_{i=1}^{k}L(h_i),
$$
which proves the inductive step.
\end{proof}

Independently, S. Kakutani \cite{kakutani} gave a different and more complicated proof in which the assumption $V_n^3\subseteq V_{n+1}$ is weakened to the optimal condition $V_n^2\subseteq V_{n+1}$. This strengthening is vital for the proof of Theorem \ref{thm:minimal}.

Lemma \ref{lemma:birkhoff} allows us to establish the following classical  characterisation of the left uniformity.
\begin{theorem}\label{weil}
For every topological group we have 
$$
\ku U_L=\bigcup\big\{\ku U_d\del d\text{ is a continuous left-invariant pseudometric on }G\big\}.
$$
\end{theorem}

Theorem \ref{weil} can of course alternatively be viewed as the definition of the left uniform structure on $G$ and therefore serves as the model for a canonical coarse structure on $G$.
\begin{definition}\label{ros}\cite{coarsebook}
The {\em left coarse structure} $\ku E_L$ on a topological group $G$ is defined by
$$
\ku E_L=\bigcap\big\{\ku E_d\del d\text{ is a continuous left-invariant pseudometric on }G\big\}.
$$
\end{definition}

Combining Theorem \ref{weil} and Definition \ref{ros} with the equivalences given in (\ref{uniformpseudo}) and (\ref{coarsepseudo}), we find that, for a subset $E\subseteq G\times G$, 
\begin{equation}\label{uniformpseudo2}
E\in \ku U_L\quad\equi\quad \e d\; \inf_{(x,y)\notin E}\;d(x,y)>0,
\end{equation}
whereas
\begin{equation}\label{coarsepseudo2}
E\in \ku E_L\quad\equi\quad \a d\; \sup_{(x,y)\in E}\;d(x,y)<\infty,
\end{equation}
where $d$ ranges over continuous left-invariant pseudometrics on $G$. 


\section{Metrisability of  coarse structure}
With Definition \ref{ros} in hand, we may ask if, in similarity with the Birkhoff--Kakutani theorem (Theorem \ref{birkhoff-kakutani}),  we may characterise metrisability of the coarse structure on the group. Although this is possible in general, the best description is available for a restricted class of topological groups.

\begin{definition}
A topological space $X$ is {\em Polish} if it is separable and the topology can be induced by a complete metric. A topological group is {\em Polish} if it is Polish as a topological space.
\end{definition}

Observe that, by the Birkhoff--Kakutani Theorem, every Polish group admits a compatible {\em left-invariant} metric. Simultaneously, it also admits a compatible {\em complete} metric. However, only some Polish groups admits compatible metrics that are left-invariant and complete. These are the so called {\em Weil complete} Polish groups, namely, those that are complete in the left uniformity $\ku U_L$.

The characterisation of metrisability involves a notion of bounded sets.
\begin{definition}
A subset $A$ of a coarse space $(X,\ku E)$ is {\em (coarsely) bounded} in case $A\times A\in E$.
\end{definition}
Thus, a subset $A$ of a topological group $G$ is bounded if and only if 
$$
\sup_{x,y\in A}d(x,y)<\infty
$$
for every continuous left-invariant pseudometric $d$ on $G$. 

Using Lemma \ref{lemma:birkhoff} we may easily establish the following characterisation.

\begin{lemma}\label{char OB}\cite{coarsebook}
The following conditions are equivalent for a subset $A$ of a topological group $G$.
\begin{enumerate}
\item $A$ is coarsely bounded,
\item for every continuous left-invariant pseudometric $d$ on $G$,
$$
{\sf diam}_d(A)<\infty,
$$
\item for every continuous isometric action  on a metric space $G\curvearrowright (X,d)$ and every $x\in X$, we have
$$
{\sf diam}_d(A\cdot x)<\infty,
$$
\item for every increasing exhaustive sequence $V_1\subseteq V_2\subseteq \ldots\subseteq G$ of open subsets with $V_n^2\subseteq V_{n+1}$, we have $A\subseteq V_n$ for some $n$.
\end{enumerate}
\end{lemma}

From Lemma \ref{char OB} it follows immediately that the class of bounded sets in a topological group $G$ forms an ideal of subsets containing all singletons and  that
$$
\ov A,\quad AB, \quad A\inv
$$
are bounded whenever $A,B\subseteq G$ are bounded.

Observe that, if $\{x_1, x_2, \ldots\}$ is a countable dense subset of a topological group $G$ and $V$ is an open identity neighbourhood, then 
$$
V_n=\big(\{x_1,\ldots, x_n\}\cdot V\big)^{2^n}
$$
defines an increasing exhaustive sequence of open subsets of $G$ so that $V_n^2\subseteq V_{n+1}$. Applying Lemma \ref{char OB} to this sequence, we finally get a compact description of boundedness.

\begin{proposition}\cite{coarsebook}
Let $G$ be a Polish group. Then a subset $A$ is bounded if and only if, for every identity neighbourhood $V$, there is a finite set $F\subseteq G$ and a power $k\geqslant 1$ so that 
$$
A\subseteq (FV)^k.
$$
\end{proposition}

We are now ready for the characterisation of the metrisability of $\ku E_L$. For this, let us say that a topological group is {\em locally bounded} if it has a bounded identity neighbourhood or, equivalently, every point has a bounded neighbourhood.

\begin{theorem}\label{locally (OB)}\cite{coarsebook}
The following  are equivalent for a Polish group $G$.
\begin{enumerate}
\item The left-coarse structure $\ku E_L$ is metrisable,
\item $G$ is covered by a countable family of bounded sets,
\item $G$ is locally bounded, 
\item $\ku E_L$ is induced by a compatible  left-invariant metric $d$ on $G$.
\end{enumerate}
\end{theorem}

\begin{proof}
(1)$\saa$(2): Suppose that $\ku E_L$ is metrisable, that is, $\ku E_L=\ku E_d$ for some (possibly discontinuous and not necessarily left-invariant) metric on $G$. Then the sets 
$$
B_n=\{g\in G\del d(g,1)\leqslant n\}
$$
are $d$-bounded (and hence coarsely bounded) and cover $G$.

(2)$\saa$(3): Suppose $B_1,B_2,\ldots$ are bounded subsets of $G$ so that $G=\bigcup_nB_n$. Then $G=\bigcup_n\ov{B_n}$ is a countable covering of $G$ by bounded, closed subsets, so, by the Baire category theorem, some $\ov{B_n}$ must have non-empty interior. It follows that
$$
\big(\ov{B_n}\big)\inv\cdot  \ov{B_n}
$$ 
is a coarsely bounded identity neighbourhood and thus that $G$ is locally bounded.

(3)$\saa$(4): Assume that $G$ is locally bounded and pick a symmetric open bounded identity neighbourhood $V_0$. Expand this to a neighbourhood basis at the identity
$$
\ldots \subseteq V_{-2}\subseteq V_{-1}\subseteq V_0
$$
consisting of symmetric open sets so that $V_{n}^3\subseteq V_{n+1}$. Let also $\{x_1, x_2, \ldots\}$ be  a countable dense subset of  $G$ and set
$$
V_n=\big(\{x_1,\ldots, x_n\}^\pm\cdot V_0\cdot \{x_1,\ldots, x_n\}^\pm\big)^{3^n}
$$
for $n\geqslant 1$. Then all the $V_n$ are symmetric open bounded identity neighbourhoods and $V_{n}^3\subseteq V_{n+1}$ for all $n\in \Z$. Let $\ell$ be the length function obtained by applying Lemma \ref{lemma:birkhoff} and $d$ the corresponding compatible left-invariant metric on $G$, that is, 
$$
d(g,f)=\ell(g\inv f).
$$ 
By the definition of the left-coarse structure, we have $\ku E_L\subseteq \ku E_d$ and thus to see that $\ku E_L=\ku E_d$ it suffices to show that $\ku E_d\subseteq \ku E_L$. So suppose that $E\in \ku E_d$, i.e., that 
$$
\alpha=\sup_{(x,y)\in E}\ell(x\inv y)=\sup_{(x,y)\in E}d(x,y)<\infty.
$$ 
By construction, the $\ell$-ball, $B_\ell(\alpha)=\{z\in G\del \ell(z)\leqslant \alpha\}$ is contained in some $V_n$ and therefore bounded. So, if $\partial$ is a continuous left-invariant pseudometric on $G$, we have 
$$
\sup_{(x,y)\in E}\partial(x,y)\leqslant \sup_{\ell(x\inv y)\leqslant \alpha}\partial(1,x\inv y)= \sup_{z\in B_\ell(\alpha)}\partial(1,z)<\infty,
$$
whereby $E\in \ku E_\partial$. As $\partial$ is arbitrary, we conclude that $E\in \ku E_L$ and so $\ku E_d\subseteq \ku E_L$.

Finally, the implication (4)$\saa$(1) is immediate.
\end{proof}

\begin{exa}
For a simple example of a Polish group that is {\bf not} locally bounded and hence whose left coarse structure is non-metrisable, consider the infinite product $\prod_{n=1}^\infty \Z$. This has a neighbourhood basis at the identity consisting of the cylinder sets
$$
\prod_{n=1}^k\{0\}\times \prod_{n=k+1}^\infty \Z,
$$
that each have infinite diameter in the continuous invariant pseudometrics
$$
d(x,y)=|x_{k+1}-y_{k+1}|
$$
and thus fail to be bounded.
\end{exa}

It should be noted that the notion of coarse boundedness is not absolute but is instead highly dependent on the ambient group. For example, consider the semi-direct product 
$$
G=\Z\ltimes \prod_{n\in \Z} \Z.
$$
Here the open subgroup $W=\prod_{n\in \Z} \Z$ is bounded when viewed as a subgroup of $G$, but fails to be bounded when viewed as a subgroup of itself. To avoid confusion regarding this issue, we shall say that a topological group $H$ is {\em globally bounded} if it is bounded when viewed as a subgroup of itself, that is, if every continuous left-invariant pseudometric on $H$ is bounded.

\begin{definition}
A compatible left-invariant metric $d$ on a topological  group is said to be {\em coarsely proper} if $\ku E_L=\ku E_d$, i.e., if is a compatible metric for the coarse structure.
\end{definition}
One can show that $d$ is coarsely proper if and only if $d$-bounded sets are exactly the coarsely bounded sets. In other words, such $d$ give infinite diameter to every set that has infinite diameter in some continuous left-invariant pseudometric.


\section{Maximal and minimal metrics}
So far we have seen that every topological group $G$ can also be viewed as an object in the two categories
$$
{\sf Uniform}\qquad\qquad\qquad {\sf Coarse}
$$
To answer the question of whether $G$ can also be viewed in the more structured categories 
$$
{\sf LocalLipschitz}\qquad\qquad\qquad {\sf Quasimetric}
$$
note that, in that case, its uniform, respectively its coarse structure, must be metrisable. However, metrisability is not enough in itself to define a local Lipschitz or quasimetric structure as, for example, the two invariant metrics on $\R$,
$$
d_1(x,y)=|x-y|, \qquad d_4(x,y)=\sqrt{|x-y|},
$$
are both compatible and coarsely proper and thus induce the same uniform and coarse structures on $\R$. However, the two metric are neither locally Lipschitz equivalent nor quasi-isometric. Indeed, in order to isolate an inherently defined local Lipschitz structure on $G$ that is compatible with its uniform structure, we need to be able to canonically select a local bi-Lipschitz equivalence class among its compatible left-invariant metrics. A similar comment applies to quasimetric structure. For this purpose, we introduce the following classes of metrics.

\begin{definition}
A compatible left-invariant metric $d$ on a topological group $G$ is 
\begin{itemize}
\item {\em minimal} if, for every other compatible left-invariant metric $\partial$, the map
$$
(G,\partial) \overset{\sf id}\longrightarrow(G,d)
$$
is Lipschitz for short distances,
\item {\em maximal} if, for every other compatible left-invariant metric $\partial$, the map
$$
(G,d) \overset{\sf id}\longrightarrow(G,\partial)
$$
is Lipschitz for large distances.
\end{itemize}
\end{definition}

Because all minimal metrics will necessarily be locally bi-Lipschitz equivalent, they all induce the same local Lipschitz structure on $G$, which we take to be the {\em local Lipschitz structure} of $G$. Similarly, all maximal metrics are quasi-isometric and thus define the {\em quasimetric structure} of $G$.

Reiterating the discussion above, every topological group $G$ defines an object in the categories {\sf Uniform} and {\sf Coarse} and, when furthermore admitting respectively minimal and maximal metrics, it also defines objects in the categories {\sf LocalLipschitz} and {\sf Quasimetric}. Furthermore, if $\ku E_G$ and $\ku Q_G$ are the coarse and quasimetric structures thus obtained, then $\ku E_G$ is the result of applying the forgetful functor of Figure \ref{second diagram} to $\ku Q_G$. Similarly for the local Lipschitz and uniform structure.

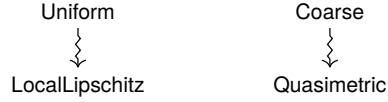
\begin{figure}
\begin{tikzcd}
&&&&{\sf Uniform}\arrow[squiggly]{d}{} & &{\sf Coarse}\arrow[squiggly]{d}{}\\
&&&&{\sf Local Lipschitz} &  & {\sf Quasimetric}\\
\end{tikzcd}
\caption{Reconstruction of small scale and large scale geometric structure}
\label{third diagram}
\end{figure}

The definitions of minimal and maximal metrics are primarily conceptual, which  unfortunately makes them rather difficult to verify and work with. We thus aim for intrinsic formulations of minimality and maximality that do not involve comparisons with the class of all compatible left-invariant metrics. For minimality, this is covered by the following result.

\begin{theorem}\label{thm:minimal}\cite{minimal}
The following are equivalent for a compatible left-invariant metric $d$ on a topological group $G$ with associated length function $\ell$.
\begin{enumerate}
\item $d$ is minimal,
\item for some identity neighbourhood $U$, constant $\eps>0$ and all $g\in G$ and $n\geqslant 1$,
$$
g, g^2,g^3, \ldots, g^n\in  U\;\;\saa\;\; \eps n\cdot \ell(g)\leqslant \ell(g^n)\leqslant n\cdot \ell(g).
$$
\end{enumerate}
\end{theorem}
In other words, the minimality of the metric can be expressed in terms of a linear growth condition of the associated length function in some identity neighbourhood of the group. For example, because in a Banach space $\norm{n\cdot x}=n\cdot\norm x$, we find that the norm metric  is a minimal metric on the associated additive topological group.

Of particular importance is the fact that a variant of Condition (2) underlies A. M. Gleason's work on Hilbert's 5th problem. Namely, the locally compact Polish groups\footnote{A locally compact group is Polish if and only if it is second countable.} admitting a metric satisfying (2) are exactly the Lie groups \cites{minimal,tao}. Combining this with the solution to Hilbert's 5th problem by Gleason, Montgomery, Zippin and Yamabe \cite{tao}, we find that the following conditions are equivalent for all locally compact Polish groups $G$.
\begin{enumerate}
\item[(a)] $G$ is a Lie group, i.e., admits a Lie group structure, 
\item[(b)] $G$ is locally euclidean, i.e., is locally homeomorphic to some $\R^n$, 
\item[(c)] $G$ is {\em NSS}, i.e., some identity neighbourhood contains no non-trivial subgroups,
\item[(d)] $G$ has a minimal metric.
\end{enumerate}

We similarly have an internal characterisation of maximality.

\begin{theorem}\cite{coarsebook}
The following are equivalent for a compatible  left-invariant metric $d$ on a topological group $G$,
\begin{enumerate}
\item $d$ is maximal,
\item $d$ is coarsely proper and {\em large scale geodesic}, that is, for some constant $K$ and all $x,y\in G$, there are $z_0=x, z_1,\ldots, z_n=y$ so that $d(z_{i-1}, z_i)\leqslant K$ and 
$$
\sum_{i=1}^nd(z_{i-1}, z_i)\leqslant K\cdot d(x,y),
$$
\item $d$ is quasi-isometric to the word metric $\rho_B$,
$$
\rho_B(x,y)=\inf\big(k\del y=xz_1\cdots z_k\text{ for some }z_1,\ldots, z_k\in B^\pm\big),
$$ 
given by a bounded generating set $B\subseteq G$.
\end{enumerate}
\end{theorem}

From condition (2), one easily gets that every outright geodesic metric is maximal and hence that  the norm induces the quasimetric structure of the additive group $(X,+)$ of a Banach space. Since the norm metric is also minimal, we see that both the local Lipschitz and quasimetric structures on $(X,+)$ are what they should be, namely, those given by the norm.

One may also use condition (3) to give a simple criterion for when, e.g., Polish groups have maximal metrics and hence canonical quasimetric structure. But first a word of caution.  Even for a Polish group, it is not true that the word metric $\rho_B$ of every bounded generating set $B\subseteq G$ will induce the quasimetric structure. But, if $B$ is either closed or if $\rho_B$ is known to be quasimetric to a compatible metric on $G$, then it does. 
\begin{theorem}\label{thm:char of max}\cite{coarsebook}
A Polish group $G$ admits a maximal metric and thus a quasimetric structure if and only if $G$ is algebraically generated by a bounded subset $B\subseteq G$. Moreover, in this case, the word metric $\rho_{\ov B}$ associated to the topological closure $\ov B$ induces the quasimetric structure.
\end{theorem}
Because the Polish groups admitting maximal metrics are exactly the ones whose left coarse structure is generated by a single entourage, these groups are called {\em monogenic}.


\section{Global Lipschitz structure}
Having responded to the problem of how to define local Lipschitz and quasimetric structure on topological groups, we are now left with the question of when these two can be promoted to a Lipschitz structure on the group. The simple answer is {\em always}.
\begin{proposition}\cite{minimal}
Suppose a topological group $G$ admits both a minimal  metric $d$ and a maximal  metric $D$. Then $G$ has one that is simultaneously minimal and maximal. Moreover, any two such metrics will be bi-Lipschitz equivalent.
\end{proposition}

\begin{proof}
Suppose first that $\partial_1$ and $\partial_2$ are both simultaneously minimal and maximal. Then, since $\partial_1$ is maximal, the map
$$
(G,\partial_1)\overset{\sf id}\longrightarrow (G,\partial_2)
$$
is Lipschitz for large distances and, since $\partial_2$ is minimal, it is also Lipschitz for short distances. It therefore follows that the map is Lipschitz. By symmetry, we see that the two metrics are Lipschitz equivalent.

To construct a simultaneously maximal and minimal metric $\partial$ from $d$ and $D$, we observe first that, since $D$ is maximal, by Theorem \ref{thm:char of max}, $G$ must be generated by a bounded set $B\subseteq G$. Let then $r>0$ be large enough so that $B$ is contained in the open $D$-ball $V$ of radius $r$ centred at the identity. Then $D$ is quasi-isometric with $\rho_V$ and the formula
$$
\partial(x,y)=\inf\Big(\sum_{i=1}^nd(v_i,1)\del x=yv_1\cdots v_n \;\&\; v_i\in V\Big)
$$
defines a compatible left-invariant metric on $G$ that is quasi-isometric to $\rho_V$ and hence also to $D$. Moreover, if $U$ is an identity neighbourhood so that $U^2\subseteq V$, then $d$ and $\partial$ agree on $U$ and hence $\partial$ is also minimal. Thus, $\partial$ is both minimal and maximal.
\end{proof}

This means that the {\em Lipschitz structure} on $G$  can be defined as that given by any compatible left-invariant metric that is simultaneously minimal and maximal. Furthermore, its existence is simply equivalent to the simultaneous existence of the local Lipschitz and the quasimetric structure. We can think of these issues in terms of Figure \ref{fourth diagram}, which reverses the reductions of Figure \ref{second diagram} to a problem of reconstruction of geometric structure.

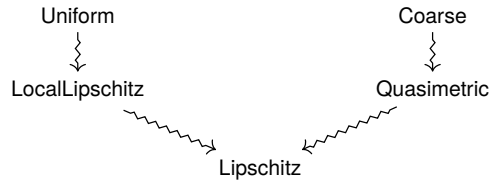
\begin{figure}
\begin{tikzcd}
&&&{\sf Uniform}\arrow[squiggly]{d}{} & &{\sf Coarse}\arrow[squiggly]{d}{}\\
&&&{\sf Local Lipschitz} \arrow[squiggly]{dr}{} &  & {\sf Quasimetric}\arrow[squiggly]{dl}{} \\
&&&&{\sf Lipschitz} &
\end{tikzcd}
\caption{Reconstruction of geometric categories}
\label{fourth diagram}
\end{figure}

As noted before, groups with minimal metrics are close to Lie groups. This is borne out by the following result due to H. Ando,  M. Doucha and Y. Matsuzawa \cite{ando}.
\begin{theorem}
Let $G$ be a connected Banach-Lie group with Banach-Lie algebra $\go g$ and define the  {\em exponential length} function ${\sf el}_G$   by the formula
$$
{\sf el}_G(g)=\inf\Big\{ \sum_{i=1}^n\norm{x_i} \Del n\geqslant 1, \; x_i\in {\go g}, \; g=\exp(x_1)\cdots\exp(x_n)\Big\}.
$$
Then the associated left-invariant metric $d(g,f)={\sf el}_G(g\inv f)$ is a compatible metric on $G$ that is simultaneously maximal and minimal and therefore defines the Lipschitz structure of $G$.
\end{theorem}

For a discussion of the converse question of whether minimal metrics are enough to construct an appropriate Lie group structure see \cite{bull}.


\section{Homeomorphism groups}

Let $M$ be a closed manifold and let
$$
{\sf Homeo}_0(M)
$$ 
be the identity component of its homeomorphism group ${\sf Homeo}(M)$ with the compact-open or, equivalently, uniform convergence topology. That is, ${\sf Homeo}_0(M)$ is the group of isotopically trivial homeomorphisms of $M$. 

Let us recall the fragmentation lemma of G. M. Fisher \cite{fisher} and  R. D. Edwards and R. C. Kirby \cite{kirby}. This states that, if 
$$
\ku U=\{U_1,\ldots, U_n\}
$$ 
is an open cover of $M$, then the set
$$
V=\{g\in {\sf Homeo}_0(M)\del g=h_1\cdots h_n \text{ for some }h_1,\ldots, h_n \text{ with }{\sf supp}(h_i)\subseteq U_i\},
$$
where ${\sf supp}(h)=\{x\in M\del h(x)\neq x\}$,
is an identity neighbourhood in ${\rm Homeo}_0(M)$. Because ${\sf Homeo}_0(M)$ is connected, it is generated by $V$, that is,
$$
{\sf Homeo}_0(M)=\bigcup_{k=1}^\infty V^k.
$$

So, it follows that we may define the {\em fragmentation metric} on ${\sf Homeo}_0(M)$ associated with $\ku U$ as the word metric
$$
\rho_{\ku U}(g,f)=\min (k\del f=gh_1\cdots h_k\;\&\; h_i\in V^\pm).
$$

For the statement of the next theorem, given a closed $m$-dimensional manifold $M$, a set $U\subseteq M$ is called an  {\em embedded open ball} if there is a homeomorphic embedding $B(2)\maps \phi M$ of the open $2$-ball $ B(2)\subseteq \R^m$, which maps the open $1$-ball $ B(1)\subseteq \R^m$ onto $U$.

\begin{theorem}\cite{katie}
Let $M$ be a closed manifold. Then ${\sf Homeo}_0(M)$ is monogenic and, if $\ku U=\{U_1, U_2, \ldots, U_n\}$ is a covering of $M$ by embedded open balls, the associated fragmentation metric $\rho_{\ku U}$ metrises the quasimetric structure on ${\sf Homeo}_0(M)$.
\end{theorem}

\begin{proof}
Let $\ku U=\{U_1, U_2, \ldots, U_k\}$ be given and let $V$ be defined as above. Set
$$
V_i=\{g\in {\sf Homeo}_0(M)\del {\sf supp}(g)\subseteq U_i\},
$$
whereby $V\subseteq V_1\cdots V_n$. We claim that each $V_i$ and thus also $V$ is bounded.  

To see this, let $O$ be an arbitrary identity neighborhood in ${\sf Homeo}_0(M)$ and fix some small open set $W\subseteq U_i$ such that any homeomorphism $f$ with ${\sf supp}(f)\subseteq W$  is contained in $O$.   Choose also a homeomorphism $h \in {\sf Homeo}_0(M)$ such that $h[U_i]\subseteq W$.  Then, for all $g\in V_i$,
$$
{\sf supp}(hgh\inv)=h\cdot {\sf supp}(g)\subseteq h[U_i]\subseteq W
$$ 
and so 
$$
V_i\subseteq h\inv Oh.
$$
Because $O$ was arbitrary, this shows that $V_i$ is bounded.

Because $V$ generates ${\sf Homeo}_0(M)$, it follows that ${\sf Homeo}_0(M)$ is monogenic and that the associated word metric $\rho_{\ku U}$ metrises the quasimetric structure on ${\sf Homeo}_0(M)$.
\end{proof}

Suppose now that $S$ is a closed orientable surface with universal cover $\tilde S\overset \pi\longrightarrow S$. Assume also that $d$ is a compatible proper metric on $\tilde S$ invariant under deck-transformations and that $D\subseteq \tilde S$ is a compact fundamental domain, that is, $\pi[D]=S$. Define for $f\in {\sf Homeo}_0(S)$
$$
\ell(f)={\sf diam}_d\big(\tilde f[D]\big),
$$
where $\tilde f$ is some lift of $f$ to $\tilde S$. Then E. Militon \cite{militon} shows that there is a constant $K$ so that
$$
\frac 1K \ell(g\inv f) -K\leqslant \rho_{\ku U}(g,f) \leqslant K \ell(g\inv f)+K
$$
for all $g,f\in {\sf Homeo}_0(S)$. In other words, the fragmentation metric can be understood in terms of the maximal displacement on the universal cover.

\begin{exa}
For every $n\geqslant 1$, the homeomorphism group
$$
{\sf Homeo}(S^n)
$$
of the $n$-dimensional sphere is globally bounded \cite{OB}. On the other hand, if $M$ is a compact manifold of dimension at least $2$ with infinite fundamental group, then there is an isomorphic coarse embedding 
$$
C([0, 1]) \to {\sf  Homeo}_0(M)
$$
and thus ${\sf Homeo}_0(M)$ is coarsely universal among all separable metric spaces \cite{katie}.
It is completely open what happens for other closed manifolds with finite fundamental group.
\end{exa}


\section{Non-Archimedean Polish groups}
Consider a closed manifold $M$. The study of its group ${\sf Homeo}(M)$ of homeomorphisms to some extent splits into two very different topics, namely, the structure of the identity component ${\sf Homeo}_0(M)$, which is an open normal subgroup of ${\sf Homeo}(M)$, and then the structure of the {\em mapping class group} 
$$
{\sf MCG}(M)={}^{{\sf Homeo}(M)}/{}_{{\sf Homeo}_0(M)}.
$$
Because ${\sf Homeo}_0(M)$ is open in ${\sf Homeo}(M)$, the mapping class group ${\sf MCG}(M)$ is countable and thus amenable to the standard techniques of geometric group theory. 

However, if one passes to non-compact manifolds $M$, such as the $2$-sphere with a Cantor set of punctures or an infinite-genus surface, this nice setup breaks down. Indeed, the identity component ${\sf Homeo}_0(M)$ need no longer be open, but only closed, and hence the mapping class group may be an uncountable Polish group. Nevertheless, the mapping class group may still be analysed using some of the same tools as before. Concretely, suppose for simplicity that $M$ is a connected orientable surface without boundary and let $C(M)$ be the {\em curve graph} of $M$, whose vertices are the isotopy classes of essential simple closed curves in $M$ and where two such classes are connected by an edge in $C(M)$ provided they admit disjoint realisations in $M$. Let also 
$$
{\sf MCG}_+(M)={}^{{\sf Homeo}_+(M)}/{}_{{\sf Homeo}_0(M)}
$$
denote the group of orientation-preserving mapping classes. Then ${\sf MCG}_+(M)$ acts canonically on $C(M)$ and, as shown independently by J. Bavard, S. Dowdall and K. Rafi \cite{bavard} and J. Hern\'andez Hern\'andez, I. Morales and F. Valdez \cite{hernandez}, this action induces an isomorphism of topological groups
$$
{\sf MCG}_+(M) \iso {\sf Aut}(C(M)),
$$
where the latter is equipped with the so-called {\em permutation group topology}. 

To introduce the latter topology, we consider a more general setting. Namely, the context of our discussion below is models of first-order logic, that is, structures of the form ${\ku A}=\langle A, \{\phi_i\}_{i\in I}, \{R_j\}_{j\in J}\rangle$, where $A$ is some set, each $\phi_i\colon A^{k_i}\to A$ is a function of some finite number $k_i$ of variables and $R_j\subseteq A^{k_j}$ is a relation of finite arity $k_j$. The {\em automorphism group} of $\ku X$ is then the group ${\sf Aut}(\ku A)$ of all permutations $g$ of $A$ that commute with the functions $\phi_i$,
$$
g\big(\phi_i(a_1,\ldots, a_{k_i})\big)=\phi_i\big(g(a_1),\ldots,g(a_{k_i})\big)
$$
and preserve the relations $R_j$, 
$$
(a_1,\ldots, a_{k_j})\in R_j\equi \big(g(a_1),\ldots,g(a_{k_j})\big)\in R_j.
$$
The permutation group topology on ${\sf Aut}(\ku A)$ is the topology induced by the inclusion ${\sf Aut}(\ku A)\subseteq A^A$ where $A$ is given the discrete topology. In other words, this is the topology obtained by declaring pointwise stabilisers 
$$
W_a=\big\{g\in {\sf Aut}(\ku A)\del g(a)=a\big\}
$$
to be open. 

A topological group $G$ is said to be {\em non-Archimedean} if it has a neighbourhood basis at the identity consisting of open subgroups. One may easily show that a Polish group $G$ is non-Archimedean if and only if it isomorphic to the automorphism group ${\sf Aut}(\ku A)$ for some countable structure $\ku A$. In fact, $\ku A$ may even be taken to be a countable graph. The mapping class groups ${\sf MCG}_+(M)$ above are thus one particular source of interesting non-Archimedean Polish groups.
To get an understanding of their geometry, we investigate their actions on graphs.


\section{Actions on graphs}
Suppose $G$ is a topological group acting continuously by automorphisms on a connected graph $\ku X$. For our modest purposes, we may simply view $\ku X$ as a first order structure $\ku X=\langle V,E\rangle$ consisting of a non-empty set $V$ of vertices equipped with a symmetric edge relation $E\subseteq V\times V$. Thus, an action of $G$ by automorphisms on $\ku X$ can be viewed as an action $G\curvearrowright V$ by permutations so that $E$ is invariant under the diagonal action  $G\curvearrowright V\times V$,  
$$
g\cdot(v,w)=(gv,gw).
$$
In particular, this induces an action $G\curvearrowright E$. Furthermore, continuity of the action means that the associated mapping $G\to {\sf Aut}(\ku X)$ is continuous, that is,  that each vertex stabiliser
$$
G_v=\{g\in G\del gv=v\}
$$
is an open subgroup of $G$. Observe also that 
$$
G_{fv}=fG_vf\inv
$$
for all $v\in V$ and $f\in G$.

The action $G\curvearrowright \ku X$ is said to be {\em cofinite} if the quotient spaces  $G\setminus V$ and $G\setminus E$ are finite. Because the graphs we consider are connected, every vertex belongs to an edge, which means that the action is cofinite if and only if just $G\setminus E$ is finite. Similarly, because $\ku X$ is connected, the shortest path distance in $\ku X$ defines a metric  $\rho$  on the vertex set $V$. Finally, for every edge $e=(v,w)\in E$, we let $o(e)=v$ and $t(e)=w$.

\begin{theorem}\label{action}
Let  $G$ be a topological group acting continuously by automorphisms on a connected graph $\ku X=\langle V,E\rangle$. Assume also that the action is cofinite and that the vertex stabilisers $G_w$ are bounded in $G$. Then $G$ is monogenic and, for any vertex $v\in V$, the evaluation map 
$$
g\in G\mapsto gv\in V
$$
is a  quasi-isometry between $G$ and the metric space $(V,\rho)$.
\end{theorem}

\begin{proof}
Fix $v\in V$. Since $G$ acts continuously by isometries on the large scale geodesic metric space $(V,\rho)$, if we can show that the action is {\em coarsely proper}, that is, that, for every $m\geqslant 1$, the set 
$$
\{g\in G\del \rho (gv,v)\leqslant m\}
$$
is bounded in $G$, then the evaluation map $g\mapsto gv$ defines a coarse equivalence by the Milnor--Schwartz lemma \cite[Theorem 2.77]{coarsebook}.

Because the action $G\curvearrowright \ku X$ is cofinite, we may find finite transversals  $T\subseteq V$ and $S\subseteq E$ for the corresponding $G$-actions. Without loss of generality, $v\in T$. Note also that, as every edge can be mapped via $G$ to an edge $e$ such that $o(e)\in T$, we may suppose that $S$ is chosen so that $o(e)\in T$ for all $e\in S$. For all $e,e'\in S$ such that
$$
G\cdot t(e)=G\cdot o(e')
$$
pick also some $f\in G$ such that 
$$
t(e)=f\cdot o(e').
$$
and let $F\subseteq G$ be the finite collection of these $f$. We set 
$$
A=F\cup \bigcup_{u\in T}G_u,
$$ 
which is bounded in $G$.

Assume also that $g\in G$ with $\rho(gv,v)=n$. Then we may find some $e_1,\ldots,e_n\in S$ and $h_1,\ldots, h_n\in G$ such that 
$$
h_1e_1,\ldots, h_ne_n
$$ 
form an edge path in $\ku X$ from $o(h_1e_1)=v$ to $t(h_ne_n)=gv=o(gh_1e_1)$. In particular, $G\cdot t(e_i)=G\cdot o(e_{i+1})$ for all $i<n$ and hence  there are $f_1,\ldots, f_{n-1}\in F$ so that
$$
t(e_i)=f_{i}\cdot o(e_{i+1})
$$ 
and thus also
\maths{
h_if_{i}\cdot o(e_{i+1})
=h_i\cdot t(e_{i})
=t(h_ie_i)
=o(h_{i+1}e_{i+1})
=h_{i+1}\cdot o(e_{i+1})
}
for all $i<n$. It therefore follows that $h_{i+1}\in h_if_{i}G_{o(e_{i+1})}$ for all $i<n$ and thus that
\maths{
h_n
&\in h_{n-1}f_{n-1}G_{o(e_{n})}\\
&\subseteq h_{n-2}f_{n-2}G_{o(e_{n-1})}f_{n-1}G_{o(e_{n})}\\
&\subseteq \ldots\\
&\subseteq h_{1}f_{1}G_{o(e_{2})}\cdots f_{n-2}G_{o(e_{n-1})}f_{n-1}G_{o(e_{n})}.
}
Note also that, as $e_1\in S$, we have both $o(e_1)\in T$ and $h_1\cdot o(e_1)=o(h_1e_1)=v\in T$, whereby
$$
h_1\cdot o(e_1)=o(e_1).
$$ 
So $h_1\in G_{o(e_1)}$ and 
\mathseq{eq1}{
h_n
&\in h_{1}f_{1}G_{o(e_{2})}\cdots f_{n-2}G_{o(e_{n-1})}f_{n-1}G_{o(e_{n})}\\
&\subseteq G_{o(e_{1})}f_{1}G_{o(e_{2})}\cdots f_{n-2}G_{o(e_{n-1})}f_{n-1}G_{o(e_{n})}\\
&\subseteq A^{2n-1}.\\
}
Finally, since $gh_1\cdot o(e_1)=gv=h_n\cdot t(e_n)$, there is some $f_n\in F$ so that $t(e_n)=f_n\cdot o(e_1)$. It thus follows that
$$
h_nf_n\cdot o(e_1)=h_n\cdot t(e_n)=t(h_ne_n)=gv=g\cdot o(e_1),
$$
whereby $g\in h_nf_nG_{o(e_1)}\subseteq A^{2n+1}$.  

Because $n$ was arbitrary, this shows that, for all $m$,
$$
\big\{g\in G\del \rho(gv,v)\leqslant m\big\}\subseteq \bigcup_{n\leqslant m}A^{2n+1}
$$
and therefore the former set is bounded in $G$.
\end{proof}

\begin{theorem}\label{existence}
Suppose $G$ is a monogenic non-Archimedean Polish group. Then $G$ admits a cofinite continuous action on a connected graph $\ku X$ with bounded vertex stabilisers. 
\end{theorem}

\begin{proof}
Because $G$ is monogenic, i.e., admits a maximal metric, it is locally bounded by Theorem \ref{locally (OB)} and generated by a bounded subset $A\subseteq G$ by Theorem \ref{thm:char of max}. Also, because $G$ is non-Archimedean, $G$ admits a neighbourhood basis at the identity consisting of open subgroups. In particular, $G$ has an open subgroup $W\leqslant G$ that is bounded in $G$, whereby, as $A$ is bounded, there is a finite symmetric set $F\subseteq G$ and some $p$ so that $A\subseteq (FW)^p$. It thus follows that $G$ is generated by the bounded symmetric set $F\cup W$.

We define a graph $\ku X$ with vertex set $G/W$ by defining the set $E$ of edges by
$$
(gW,hW)\in E\;\equi\; Wg\inv hW= WfW \quad\text{ for some }f\in F.
$$
Then it is clear that the action of $G$ on $G/W$ by left multiplication is an action by automorphisms of the graph $\ku X$ and that $(gW,gwfW)$ is an edge for all $g\in G$, $w\in W$ and $f\in F$. In particular, if $w_1,\ldots, w_n\in W$ and $f_1,\ldots, f_n\in F$, the sequence of vertices
\maths{
&w_1f_1\cdots w_{n-1}f_{n-1}w_nf_nW, \\
&w_1f_1\cdots w_{n-1}f_{n-1}W,\\
& \ldots\\
& w_1f_1W,\\
&W
}
forms a vertex path in $\ku X$. Because $G=\bigcup_{n=0}^\infty (WF)^n$, it follows that $\ku X$ is connected. Also, as $F$ is finite, the action $G\curvearrowright \ku X$ is cofinite. Finally, note that, for every vertex $gW$ in $\ku X$, the vertex stabiliser
$$
{\sf stab}(gW)=gWg\inv
$$
is bounded in $G$. 
\end{proof}

Theorems \ref{action} and \ref{existence} provide a recipe for how to compute the quasimetric structure of a monogenic non-Archimedean Polish group $G$. Namely, by finding a bounded open subgroup $W$ and a finite symmetric set $F\subseteq G$ so that the union $F\cup W$ generates $G$. It is thus imperative to find methods for determining when an open subgroup $W\leqslant G$ is bounded in $G$. This problem is addressed in the next section.


\section{Independence relations}
Recall that a Polish group $G$ is non-Archimedean if and only if it is isomorphic to the automorphism group ${\sf Aut}(\ku A)$ for some countable first-order structure $\ku A$. So, in the following, let us fix such a structure $\ku A$ and let $\Omega$ be the collection of all finite length tuples 
$$
\ov a=(a_1,\ldots, a_n)
$$
of elements $a_i$ of the structure $\ku A$. We use $\ov a, \ov b, \ov c, \ldots$  as variables for elements of $\Omega$ and shall write $(\ov a, \ov b)$ to denote the concatenation of two tuples $\ov a $ and $\ov b$. The automorphism group ${\rm Aut}({\ku A})$ acts naturally on $\Omega$ via 
$$
g\cdot (a_1, \ldots, a_n)=(ga_1, \ldots, ga_n).
$$
With this notation, the pointwise stabiliser subgroups 
$$
W_{\ov a}=\{g\in {\rm Aut}({{\ku A}})\del g\cdot \ov a=\ov a\},
$$ 
with $\ov a\in \Omega$, form a neighbourhood basis at the identity in ${\rm Aut}({\ku A})$. An {\em orbital type} $\ku O$ in ${\ku A}$ is simply the orbit of some tuple $\ov a$ under the action of ${\rm Aut}({{\ku A}})$.  Also, we let $\OO(\ov a)$ denote the orbital type of $\ov a$ , i.e., 
$$
\OO(\ov a)={\rm Aut}({{\ku A}})\cdot \ov a.
$$
Similarly, define
$$
{\ku O}(\ov a, \ov b)=\big\{(g\ov a,g \ov b)\del g\in {\sf Aut}({\ku A})\big\}
$$
and
$$
\OO(\ov b/\ov a)=W_{\ov a}\cdot \ov b=\{g\cdot \ov b\del g\in W_{\ov a}\}.
$$
Thus, for example, the following equivalences hold for all automorphisms $g,f\in {\sf Aut}({\ku A})$,
\begin{equation}\label{double-sided}
{\ku O}(g\ov b/\ov a)={\ku O}(f\ov b/\ov a)
\;\; \equi\;\;
{\ku O}(g\ov b,\ov a)={\ku O}(f\ov b,\ov a)
\;\;\equi \;\; 
W_{\ov a}gW_{\ov b}=W_{\ov a}fW_{\ov b}.
\end{equation}

\begin{definition}
Suppose ${\ku A}$ is a countable structure and $\ov a\in \Omega$. A binary relation $\forkindep$ on $\Omega$ is said to be an {\em orbital ${\ov a}$-independence relation}\footnote{This slightly generalises the  notion of independence relations from \cite{coarse book}. The abstract study of independence relations is a central part of current model theory. } provided that it satisfies the following three properties,
\begin{enumerate}
\item(monotonicity) for all tuples $\ov b$, $\ov c$ and $\ov d$,  if $(\ov c,\ov d)\forkindep\ov b$, then also 
$$
\ov c\forkindep\ov b \quad\text{ and }\quad \ov d\forkindep\ov b,
$$
\item(existence) for all tuples $\ov b$ and $\ov c$, there is some $g\in W_{\ov a}$ so that 
$$
g\ov c\forkindep \ov b,
$$
\item(finite stationarity) for all tuples $\ov b$, the set
$$
\big \{{\ku O}(\ov c/\ov b)\del {\ku O}(\ov c/\ov a)={\ku O}(\ov b/\ov a) \;\;\&\;\; \ov c\forkindep \ov b\big\}
$$
is finite.
\end{enumerate}
\end{definition}

The relation $\ov c\forkindep\ov b$  reads as  {\em $\ov c$ is independent from $\ov b$ over $\ov a$}, which also hints at its origins.

\begin{theorem}\cite{coarsebook}
If ${\ku A}$ is a countable structure and $\forkindep$ is an orbital $\ov a$-independence relation for some finite tuple $\ov a$, then the pointwise stabiliser $W_{\ov a}$ is globally bounded.
\end{theorem}

\begin{proof}
Assume $\ov b\in \Omega$. We show that there is a finite set $F\subseteq W_{\ov a}$ so that 
$$
W_{\ov a}\subseteq W_{\ov b}\cdot F\inv\cdot W_{\ov b}\cdot F\cdot W_{\ov b},
$$
which, because the pointwise stabiliser subgroups form a neighbourhood basis at the identity, will show that $W_{\ov a}$ is globally bounded.

So, choose by finite stationarity some $\ov c_1,\ldots, \ov c_n\in \Omega$ satisfying
$$
{\ku O}(\ov c_i/\ov a)={\ku O}(\ov b/\ov a) \quad\&\quad \ov c_i\forkindep \ov b
$$
and so that, if ${\ku O}(\ov c/\ov a)={\ku O}(\ov b/\ov a)$  and $\ov c\forkindep \ov b$ for some $\ov c\in \Omega$, then 
$$
{\ku O}(\ov c/\ov b)={\ku O}(\ov c_i/\ov b)
$$
for some $i$. For each $i$, because ${\ku O}(\ov c_i/\ov a)={\ku O}(\ov b/\ov a)$, we can find $f_i\in W_{\ov a}$ so that $\ov c_i=f_i\ov b$. Let $F=\{f_1,\ldots, f_n\}$.

Suppose now that $g\in W_{\ov a}$. By existence, there is some $h\in W_{\ov a}$ so that
$$
h(\ov b,g\ov b)\forkindep \ov b,
$$
whereby, using monotonicity, we have
$$
{\ku O}(h\ov b/\ov a)={\ku O}(\ov b/\ov a)\quad\&\quad  h\ov b\forkindep \ov b
$$
and
$$
{\ku O}(hg\ov b/\ov a)={\ku O}(\ov b/\ov a)\quad\&\quad  hg\ov b\forkindep \ov b.
$$
By the choice of the $\ov c_i$ and $f_i$, it follows that
$$
{\ku O}(h\ov b/\ov b)={\ku O}(f_i\ov b/\ov b)
$$
and 
$$
{\ku O}(hg\ov b/\ov b)={\ku O}(f_j\ov b/\ov b)
$$
for some $i,j$. Thus, by Equation \ref{double-sided}, 
$$
W_{\ov b}hW_{\ov b}=W_{\ov b}f_iW_{\ov b}
$$
and 
$$
W_{\ov b}hgW_{\ov b}=W_{\ov b}f_jW_{\ov b}.
$$
We thus conclude that 
$$
g\;\in\;  h\inv W_{\ov b}f_jW_{\ov b}\;\subseteq\; W_{\ov b}F\inv W_{\ov b}FW_{\ov b}
$$
as required.
\end{proof}


\begin{exa}[The $\aleph_0$-regular tree]\label{tree}
Consider the $\aleph_0$-regular tree $\ku T_\infty$. That is, $\ku T_\infty=\langle V,E\rangle$ is a countable connected undirected graph without loops in which every vertex has infinite valence. Since $\ku T_\infty$ is a tree, it admits a natural notion of {\em convex hull}, namely, for a set of vertices $A$ and a vertex $v$, we set $v\in {\sf conv}(A)$ if there are $w,u\in A$ so that $v$ lies on the unique path from $w$ to $u$. 

Fix now any vertex $a$ in $\ku T_\infty$. For finite tuples of vertices $\ov b, \ov c$ enumerating finite subsets $B,C\subseteq V$, set 
$$
\ov b\forkindep\ov c\;\equi\; {\sf conv}(B\cup\{a\})\cap {\sf conv}(C\cup\{a\})=\{a\}.
$$
We claim that $\forkindep$ is an orbital $a$-independence relation on $\ku T_\infty$ and therefore that the stabiliser $W_a$ is globally bounded.

That $\forkindep$ is monotone is obvious. Also, because $B$ and $C$ are finite sets, then so are ${\sf conv}(B\cup\{a\})$ and ${\sf conv}(C\cup\{a\})$.  This makes it is easy to find some $g\in {\sf Aut}(\ku A)$ fixing $a$, so that 
$$
g\big[{\sf conv}(C\cup\{a\})\big]\cap {\sf conv}(B\cup\{a\})=\{a\}.
$$
Since then $g\big[{\sf conv}(C\cup\{a\})\big]={\sf conv}(gC\cup\{a\})$, one sees that $g\ov c\forkindep\ov b$, thus verifying the existence condition. 

Finally, for  finite stationarity, suppose a tuple $\ov b$ is given enumerating some finite $B\subseteq V$. Suppose also $\ov c_1$ and $\ov c_2$ are two tuples so that 
$$
\OO(\ov c_1/a)=\OO(\ov c_2/a)=\OO(\ov b/a),
$$
$\ov c_1\forkindep\ov b$ and $\ov c_2\forkindep\ov b$. This means that there is some $g\in {\sf Aut}(\ku A)$ fixing $a$ so that $g\ov c_1=\ov c_2$. Furthermore, if $C_1$ and $C_2$ are the sets enumerated by $\ov c_1$ and $\ov c_2$, then 
$$
{\sf conv}(B\cup\{a\})\cap {\sf conv}(C_1\cup\{a\})=\{a\}={\sf conv}(B\cup\{a\})\cap {\sf conv}(C_2\cup\{a\}).
$$
It is then easy to find some $f\in {\sf Aut}(\ku A)$ that fixes all of ${\sf conv}(B\cup\{a\})$ while $f\ov c_1=\ov c_2$, which, in turn, implies that 
$$
\OO(\ov c_1/\ov b)=\OO(\ov c_2/\ov b)
$$
and therefore witnesses finite stationarity.

We now see that the tautological action ${\sf Aut}(\ku T_\infty)\curvearrowright \ku T_\infty$ is a continuous action on a connected graph with (globally) bounded vertex stabilisers. Furthermore, as  ${\sf Aut}(\ku T_\infty)$ acts transitively on the set of edges, it follows that the action is cofinite and hence that, for any vertex $v$, the orbital map
$$
g\mapsto gv
$$
is a quasi-isometry between  ${\sf Aut}(\ku T_\infty)$ and the tree $\ku T_\infty$ viewed as a metric space. 
\end{exa}


\section{Functorial amalgamation in Fra\"iss\'e classes}
If ${\ku A}=\langle A, \{\phi_i\}_{i\in I}, \{R_j\}_{j\in J}\rangle$ is some first-order structure, a {\em substructure}  of $\ku A$ is simply that which is given by some subset $B\subseteq A$ closed under the functions $\phi_i$.  Conversely, in this case $\ku A$ is a {\em superstructure} of $\ku B$. Similarly, a {\em reduct} of the structure $\ku A$ is any structure of the form
$$
{\ku B}=\langle A, \{\phi_i\}_{i\in I'}, \{R_j\}_{j\in J'}\rangle,
$$
where $I'\subseteq I$ and $J'\subseteq J$. In this case, $\ku A$ is said to be an {\em expansion} of $\ku B$. Thus, substructures of a structure $\ku A$ are obtained by diminishing the universe $A$ but keeping the same functions and relations, whereas reducts are obtained by diminishing the classes of functions and relations in the structure but keeping the same universe.

\begin{definition}
A countable first-order structure $\ku A$ is {\em ultrahomogenous} if, for any two finitely generated substructures $\ku B,\ku C\subseteq \ku A$ and any isomorphism
$$
\ku B\maps f \ku C,
$$
there is an automorphism $g\in {\sf Aut}(\ku A)$ that extends $f$. 
\end{definition}

If ${\ku A}=\langle A, \{\phi_i\}_{i\in I}, \{R_j\}_{j\in J}\rangle$ is a countable first-order structure, it is quite easy to expand $\ku A$ to an ultrahomogeneous countable first-order structure 
$$
{\ku B}=\langle A, \{\phi_i\}_{i\in I}, \{R_j\}_{j\in J'}\rangle,
$$
i.e., with $J\subseteq J'$, so that every automorphism $g$ of $\ku A$ is also an automorphism of $\ku B$, that is so that
$$
(a_1,\ldots, a_{k_j})\in R_j\equi \big(g(a_1),\ldots,g(a_{k_j})\big)\in R_j
$$
for all $j\in J'\setminus J$ and all tuples $(a_1,\ldots, a_{k_j})$ in $A$. To do this, one simply expands the sequence $\{R_j\}_{j\in J}$ with the relations given by the orbital types $\ku O(\ov a)\subseteq A^n$ for all tuples $\ov a=(a_1,\ldots, a_{n})$ in $A$. Because, as noted before, every non-Archimedean Polish group $G$ is isomorphic to the automorphism group ${\sf Aut}(\ku A)$ of some countable structure, this argument shows that one can furthermore assume this structure $\ku A$ to be ultrahomogeneous. 

Although the above construction is completely explicit, it still requires one to be able to compute the orbital types $\ku O(\ov a)$ of tuples $\ov a$ in $\ku A$. Nevertheless, ultrahomogeneity is quite frequent among first-order structures and we shall restrict ourselves to this setting.

A general approach to the study of countable ultrahomogeneous structures is what is now termed {\em Fra\"iss\'e classes}. Such classes permits one to represent ultrahomogeneous structures as limits of finitely generated structures in a precise sense.
To explain this, let us define the  {\em signature} $L$ of the first-order structure
$$
{\ku A}=\langle A, \{\phi_i\}_{i\in I}, \{R_j\}_{j\in J}\rangle
$$
to be the two indexed sequences $\{k_i\}_{i\in I}$ and  $\{k_j\}_{j\in J}$. It thus makes sense to talk of {\em isomorphisms} of two structures of the same signature $L$, namely, as bijections of the underlying sets that conjugate the functions of the same index $i$ and preserve relations of the same index $j$. An {\em embedding} of a structure $\ku A$ into a structure $\ku B$ of the same signature is defined to be an isomorphism between $\ku A$ and a substructure of $\ku B$.

\begin{definition}
A \index{Fra\"iss\'e class}{\em Fra\"iss\'e class} is a class $\mathfrak K$ of finitely generated structures of the same signature so that 
\begin{enumerate}
\item $\mathfrak K$ contains only countably many isomorphism types,
\item (hereditary property) if ${\ku A}\in \mathfrak K$ and $\ku B$ is a finitely generated structure embeddable into ${\ku A}$, then ${\ku B}\in \mathfrak K$,
\item (joint embedding property) for all ${\ku A}, {\ku B}\in \mathfrak K$, there some ${\ku C}\in \mathfrak K$ into which both ${\ku A}$ and ${\ku B}$ embed, 
\item (amalgamation property) if ${\ku A},{\ku B}_1, {\ku B}_2\in \mathfrak K$ and ${\ku A}\maps{\eta_i}{\ku B}_i$ are embeddings, then there is some ${\ku C}\in \mathfrak K$ and embeddings ${\ku B}_i\maps{\zeta_i}   {\ku C}$ so that $\zeta_1\circ \eta_1=\zeta_2\circ\eta_2$.
\end{enumerate}
\end{definition}

\begin{figure}
\vspace{-.5cm}
\begin{tikzcd}
&&&{\ku C} &                 &     &&               &{\ku C} &\\
&&{\ku B}_1 \arrow[dashed]{ur}{}  & &{\ku B}_2\arrow[dashed]{ul}{} &&&{\ku B}_1 \arrow[dashed]{ur}{\zeta_1}  & &{\ku B}_2\arrow[dashed,swap]{ul}{\zeta_2}\\
&&&&&&&&{\ku A}\arrow[]{ul}{\eta_1}\arrow[swap]{ur}{\eta_2}&
\end{tikzcd}
\caption{Commutative diagrams for joint embedding and amalgamation properties}
\label{fraisse}
\end{figure}
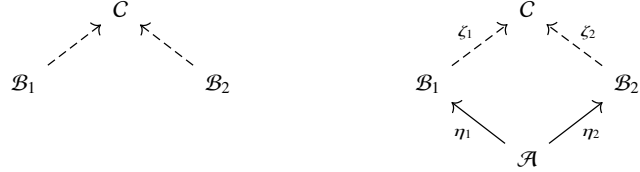

Observe first that, if ${\ku K}$ is any structure, then the class
$$
{\sf Age}(\ku K)=\big \{ \ku A\del \ku A \text{ is a finitely generated structure embeddable in }\ku K\big\}
$$ 
satisfies the hereditary and joint embedding properties. Moreover, if $\ku K$ is ultrahomogeneous, then ${\sf Age}(\ku K)$ also satisfies the amalgamation property and, finally,  if furthermore $\ku K$ has countable signature, then ${\sf Age}(\ku K)$ only contains countably many isomorphism types and is thus a Fraïssé class.

Conversely, the fundamental theorem of R. Fra\"iss\'e \cite[Theorem 6.1.2]{hodges} states that, for every Fra\"iss\'e class $\mathfrak K$ in a countable signature, there is a countable ultrahomogeneous structure $\ku K$, called the {\em Fra\"iss\'e limit} of $\mathfrak K$, so that ${\rm Age}(\ku K)=\mathfrak K$. Moreover, this limit $\mathfrak K$ is unique up to isomorphism.

We will now proceed to formulate a simple and common criterion on Fra\"iss\'e classes that implies that the resulting limit admits an orbital independence relation. This criterion has to do with when the amalgamation in the Fra\"iss\'e class can be performed in an appropriately canonical fashion. One rendering of this is given by K. Tent and M. Ziegler \cite[Example 2.2]{tent}. However, for our purposes, their notion is too weak and we instead need to impose a condition of functoriality.

\begin{definition}\label{funct amal}
Suppose $\mathfrak K$ is a Fra\"iss\'e class and that ${\ku A} \in \mathfrak K $. We say that $\mathfrak K $ admits a {\em functorial amalgamation}\footnote{Again this generalises the definition of functorial amalgamation from \cite{coarsebook}.} over $\ku A$ if there is map $\Theta$ that to all pairs of embeddings ${\ku A}\maps{\eta_i} {\ku B}_i$ with ${\ku B}_i\in\mathfrak K $,  associates a pair of embeddings 
$$
\Theta\big({\ku A}\maps{\eta_1} {\ku B}_1, {\ku A}\maps{\eta_2} {\ku B}_2\big)=\big({{\ku B}_1}\maps{\zeta_1} {\ku C}, {{\ku B}_2}\maps{\zeta_2} {\ku C}\big)
$$
into another structure ${\ku C}\in \mathfrak K $ so that 
$\zeta_1\circ\eta_1=\zeta_2\circ\eta_2$. Furthermore, we demand that, if 
$$
\Theta\big({\ku A}\maps{\eta'_1} {\ku B}'_1, {\ku A}\maps{\eta_2} {\ku B}_2\big)
=\big({{\ku B}'_1}\maps{\zeta'_1} {\ku C}', {{\ku B}_2}\maps{\zeta'_2} {\ku C}'\big)
$$
for some other embedding ${\ku A}\maps{\eta'_1} {\ku B}'_1$ into ${\ku B}_1'\in \mathfrak K$ and if ${\ku B}_1\maps{\iota}  {\ku B}_1'$ is an embedding with $\iota\circ \eta_1=\eta_1'$, then there is an embedding ${\ku C}\maps\sigma  {\ku C}'$ so that $\sigma\circ \zeta_1=\zeta_1'\circ\iota$. 
\end{definition}

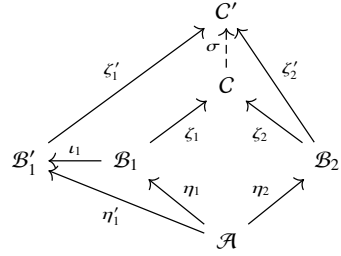
\begin{figure}
\vspace{-.5cm}
\begin{tikzcd}
&&&&&&{\ku C}'  &&\\
&&&&&&{\ku C} \arrow[dashed]{u}{\sigma} &&\\
&&&&{\ku B}'_1\arrow[]{uurr}{\zeta_1'}&{\ku B}_1 \arrow[swap]{l}{\iota_1}\arrow[swap]{ur}{\zeta_1}  & &{\ku B}_2\arrow[]{ul}{\zeta_2}\arrow[swap]{uul}{\zeta'_2}&\\
&&&&&&{\ku A}\arrow[swap]{ul}{\eta_1}\arrow[]{ur}{\eta_2}\arrow[]{ull}{\eta'_1}&&
\end{tikzcd}
\caption{Commutative diagram for functorial amalgamation}
\label{functorial amalgamation scheme}
\end{figure}

With the concept of functorial amalgamation in hand, we can now use this to define a corresponding notion of orbital independence. For this, if $\ov a$ is a finite tuple in a structure $\ku K$, let  
$$
\ku A=\langle\ov a\rangle
$$
denote the substructure of $\ku K$ generated by $\ov a$.

\begin{definition}\label{functorial indep}
Suppose $\mathfrak K$ is a Fra\"iss\'e class with limit $\ku K$,  $\ov a$ is a finite tuple in $\ku K$ and $\Theta$ is a functorial amalgamation on $\mathfrak K$ over ${\ku A}=\langle\ov a\rangle$. For finite tuples $\ov b_1, \ov b_2$ in $\ku K$, structures ${\ku B}_i=\langle \ov a,\ov b_i\rangle$ and  ${\ku D}=\langle \ov a,\ov b_1,\ov b_2\rangle$, we set 
$$
\ov b_1\forkindep[\Theta] \ov b_2
$$
if and only if 
$$
\Theta\big({\ku A}\maps{{\sf id}_\ku A} {\ku B}_1, {\ku A}\maps{{\sf id}_\ku A} {\ku B}_2\big)
=\big({{\ku B}_1}\maps{\pi\circ {\sf id}_{{\ku B}_1}} {\ku C}, {{\ku B}_2}\maps{\pi\circ {\sf id}_{{\ku B}_2}} {\ku C}\big).
$$
for  some structure $\ku C\in \mathfrak K$ and some embedding ${\ku D}\maps\pi  {\ku C}$.
\end{definition}

Alternatively, if we let 
$$
\Theta\big({\ku A}\maps{{\sf id}_\ku A} {\ku B}_1, {\ku A}\maps{{\sf id}_\ku A}{\ku B}_2\big)
=\big({{\ku B}_1}\maps{\zeta_1} {\ku C}, {{\ku B}_2}\maps{\zeta_2} {\ku C}\big)
$$
then the definition of $\forkindep[\Theta]$ can be expressed in terms of a commutative diagram.
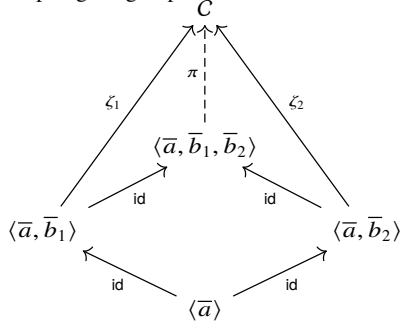
\begin{figure}
\vspace{-.5cm}
\begin{tikzcd}
&&&&&&{\ku C} &            \\
&&&&&&{}&\\
&&&&&&\langle \ov a,\ov b_1, \ov b_2\rangle\arrow[dashed]{uu}{\pi}  &            \\
&&&&&\langle \ov a,\ov b_1\rangle \arrow[swap]{ur}{\sf id}\arrow[]{uuur}{\zeta_1}   & &\langle \ov a,\ov b_2\rangle\arrow[]{ul}{\sf id}\arrow[swap]{uuul}{\zeta_2} \\
&&&&&&\langle \ov a\rangle\arrow[]{ul}{\sf id}\arrow[swap]{ur}{\sf id}&
\end{tikzcd}
\caption{Commutative diagram for orbital independence}
\label{}
\end{figure}

It now remains to check that this is indeed an orbital independence relation.

\begin{theorem}\label{functorial amalgamation OB}\cite{coarsebook}
Suppose $\mathfrak K$ is a Fra\"iss\'e class with limit $\ku K$, $\ov a$ is a finite tuple in $\ku K$ and $\Theta$ is a functorial amalgamation on $\mathfrak K$ over $\langle\ov a\rangle$.  Then $\forkindep[\Theta]$  is an orbital $\ov a$-independence relation on $\mathfrak K$. 
\end{theorem}

\begin{proof}
Monotonicity of $\forkindep[\Theta]$ follows from  functoriality of $\Theta$ and is easiest checked by a little diagram chasing. Also, the existence condition on $\forkindep[\Theta]$ follows from the ultrahomogeneity of $\ku K$ and the realisation of the amalgam $\Theta$ inside of $\ku K$. 

For finite stationarity, suppose that $\ov b$, $\ov c$ and $\ov d$  are finite tuples so that  $\ov c\forkindep[\Theta] \ov b$, $\ov d\forkindep[\Theta] \ov b$ and 
$$
\OO(\ov c/\ov a)=\OO(\ov d/\ov a).
$$
It then suffices to show that
$$
\OO(\ov c/\ov b)=\OO(\ov d/\ov b).
$$

Note that because  $\ov c\forkindep[\Theta] \ov b$ and  $\ov d\forkindep[\Theta] \ov b$, there are structures $\ku C,\ku D\in \mathfrak K$ and embeddings
$$
\langle \ov a,\ov b, \ov c\rangle\maps{\pi_{\ku C}}  {\ku C}, \quad \langle \ov a,\ov b, \ov d\rangle\maps{\pi_{\ku D}}  {\ku D},
$$
so that
$$
\Theta\big(
\langle \ov a\rangle\maps{{\sf id}}\langle \ov a,\ov b\rangle,
\langle \ov a\rangle\maps{{\sf id}}\langle \ov a, \ov c\rangle
\big)
=\big(
\langle \ov a,\ov b\rangle\maps{{\pi}_{\ku C}} {\ku C},
\langle \ov a,\ov c\rangle\maps{{\pi}_{\ku C}} {\ku C}\big)
$$
and
$$
\Theta\big(
\langle \ov a\rangle\maps{{\sf id}}\langle \ov a,\ov b\rangle,
\langle \ov a\rangle\maps{{\sf id}}\langle \ov a, \ov d\rangle
\big)
=\big(
\langle \ov a,\ov b\rangle\maps{{\pi}_{\ku D}} {\ku D},
\langle \ov a,\ov d\rangle\maps{{\pi}_{\ku D}} {\ku D}\big).
$$
Also, because $\OO(\ov c/\ov a)=\OO(\ov d/\ov a)$, there is an isomorphism $\langle \ov a,\ov c\rangle\maps\iota\langle \ov a,\ov d\rangle$ so that $\iota(\ov a)=\ov a$. Thus, as $\Theta$ is a functorial amalgamation over $\langle \ov a\rangle$, it follows that there is an embedding $\ku C\maps \sigma\ku D$ so that
$$
\sigma\circ {\pi}_{\ku C}(x)={\pi}_{\ku D}\circ\iota(x)
$$
for all $x\in \langle \ov a,\ov c\rangle$ and 
$$
\sigma\circ {\pi}_{\ku C}(x)={\pi}_{\ku D}(x)
$$
for all $x\in \langle \ov a,\ov b\rangle$. In particular, 
$$
\sigma{\pi}_{\ku C}(\ov a,\ov b,\ov c)={\pi}_{\ku D}(\ov a,\ov b, \ov d)
$$
and so 
$$
\sigma{\pi}_{\ku C}\big[\langle \ov a,\ov b,\ov c\rangle\big]={\pi}_{\ku D}\big[\langle \ov a,\ov b,\ov d\rangle\big].
$$
Thus, if we let $\rho$ be the inverse to the isomorphism
$$
\langle \ov a,\ov b,\ov d\rangle\maps{{\pi}_{\ku D}} {\pi}_{\ku D}\big[\langle \ov a,\ov b,\ov d\rangle\big]
$$
then $\rho\sigma{\pi}_{\ku C}$ is an isomorphism between $\langle \ov a,\ov b,\ov c\rangle$ and $\langle \ov a,\ov b,\ov d\rangle$
that fixes $\ov a$ and $\ov b$, while mapping $\ov c$ to $\ov d$. By  ultrahomogeneity of $\ku K$, this implies that there is an automorphism $g\in {\rm Aut}(\ku K)$ extending $\rho\sigma{\pi}_{\ku C}$, hence showing that
$$
\OO(\ov c/\ov b)=\OO(\ov d/\ov b)
$$
and thus verifying finite stationarity.
\end{proof}

\begin{figure}
\vspace{-.5cm}
\begin{tikzcd}
&&&{\ku C}    \arrow[dashed]{rr}{\sigma} &      &{\ku D}\arrow[dashed,bend left=40]{dd}{\rho}  &      \\
&&&&{}&&\\
&&&\langle \ov a,\ov b, \ov c\rangle  \arrow[dashed]{uu}{\pi_{\ku C}}  &&  \langle \ov a,\ov b, \ov d\rangle  \arrow[dashed]{uu}{\pi_{\ku D}}&   \\
&&&&{}&&\\
&&\langle \ov a,\ov c\rangle \arrow[dashed, bend right=15]{rrrr}{{\qquad\qquad}\iota} \arrow[]{uur}{{\sf id}}  & &\langle \ov a,\ov b\rangle\arrow[swap]{uul}{{\sf id}}\arrow[]{uur}{{\sf id}}    &&\langle \ov a,\ov d\rangle\arrow[swap]{uul}{\sf id} \\
&&&&{}&&\\
&&&&  \langle \ov a\rangle\arrow[]{uull}{\sf id}\arrow[]{uu}{\sf id}\arrow[swap]{uurr}{\sf id}&&
\end{tikzcd}
\caption{Commutative diagram for stationarity in Theorem \ref{functorial amalgamation OB}}
\end{figure}
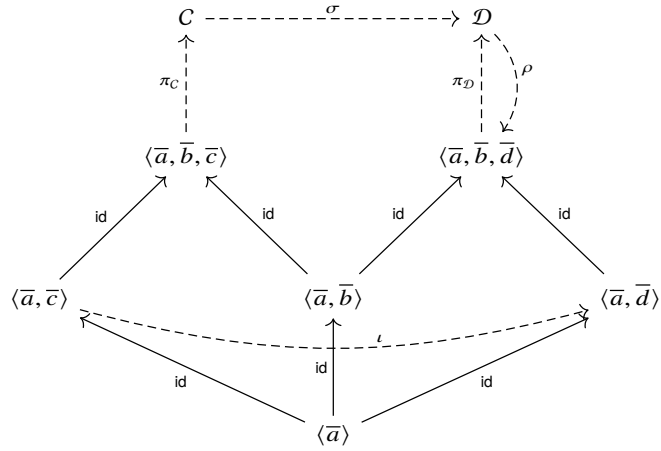

\begin{corollary}
Suppose $\mathfrak K$ is a Fra\"iss\'e class that admits a functorial amalgamation over some structure $\ku A\in \mathfrak K$. Then the automorphism group ${\sf Aut}(\ku K)$ of the Fra\"iss\'e limit $\ku K$  is locally bounded and hence admits a coarsely proper metric.
\end{corollary}

\begin{exa}[Fraïssé limits of metric spaces]
To give an easily understood example of a Fra\"iss\'e class that admits a functorial amalgamation consider the class $\mathfrak K$ consisting of all non-empty finite metric spaces where the metric only takes values in $\N$. To see how to amalgamate such spaces suppose $\ku A,\ku B_1,\ku B_2\in \mathfrak K$ and that $\ku A\maps{\eta_1}\ku B_1$ and $\ku A\maps{\eta_2}\ku B_2$ are isometric embeddings. We then consider the pseudometric space obtained by taking the disjoint union of $\ku B_1$ and $\ku B_2$ with pseudometric
$$
d(x,y)=\begin{cases}
d(x,y) & \text{ if }x,y\in \ku B_1\\
d(x,y) & \text{ if }x,y\in \ku B_2\\
\min_{z\in \ku A}d\big(x,\eta_1(z)\big)+d\big(\eta_2(z),y\big) & \text{ if }x\in \ku B_1\;\&\; y\in \ku B_2
\end{cases}
$$
Let also $\ku B_1\oplus_{\ku A}\ku B_2$ be the metric quotient of this space and define isometric embeddings $\ku B_i\maps{\zeta_i}\ku B_1\oplus_{\ku A}\ku B_2$ by $\zeta_i(x)=x$. It is straightforward to then verify that $\mathfrak K$ is a Fra\"iss\'e class and that the above is a functorial amalgamation scheme over the one-point metric space $\ku A=\{*\}$. 

The resulting Fra\"iss\'e limit $\ku K$ is commonly known as the {\em integral Urysohn metric space} and is an ultrahomogeneous countable metric space with distance set $\N$ that contains isometric copies of all finite metric spaces over $\N$. Furthermore, the automorphism group, which in this case is the isometry group ${\sf Isom}(\ku K)$, is locally bounded. In fact, for every point $x\in \ku K$, the pointwise stabiliser $W_x$ is globally bounded.

We define a graph relation $E$ on $\ku K$ by setting
$$
E=\{(x,y))\del  d(x,y)=1\}.
$$
Now, suppose $x$ and $y$ are two points in $\ku K$ and let $n=d(x,y)$. Since $\ku K$ contains all finite metric spaces with distances in  $\N$, we can find points $z_0, \ldots, z_n$  in $\ku K$ so that $d(z_0,z_n)=n$ and  $d(z_i,z_{i+1})=1$ for all $i$. By ultrahomogeneity of $\ku K$, we can then find some isometry $g\in {\sf Isom}(\ku K)$ so that $g(z_0)=x$ and $g(z_n)=y$. In particular, $d\big(g(z_i),g(z_{i+1})\big)=1$ for all $i$ and so 
$$
g(z_0), g(z_1),\ldots, g(z_n)
$$
is a vertex path in the graph $(\ku K,E)$ from $x$ to $y$. This shows that the shortest path distance $\rho(x,y)$ between any two vertices in the graph $(\ku K,E)$ equals the distance $d(x,y)$ in the metric space $\ku K$. Furthermore, $(\ku K,E)$ is  connected and ${\sf Isom}(\ku K)$ acts transitively on the set of edges $E$  and with globally bounded vertex stabilisers. 

Assembling all the information above and applying Theorem \ref{action}, we thus find that the evaluation map
$$
g\in {\sf Isom}(\ku K)\mapsto g(x)\in \ku K
$$
is a quasi-isometry between the topological group ${\sf Isom}(\ku K)$  and the metric space $\ku K$. Further details on this specific construction can by found in  \cite[Examples 6.32, 6.36]{coarsebook}.
\end{exa}


\section{Further directions for research}
The results presented here of course only provides the basic framework for the analysis of concrete examples of topological groups. Although perhaps intellectually pleasing in itself, the real value of the theory lies in imparting geometric structure to common topological groups which may then reveal further details about their structure. An important aspect of this would be to establish direct links between the algebraic, dynamical and geometric structure of topological groups. For example, the geometric structure of some mathematical object may provide information about the geometric structure of its symmetry group, which in turn may have consequences for the algebraic structure of the same. 

\begin{problem}
How are the algebraic, dynamical and geometric properties of topological groups correlated? 
\end{problem}

One area in which the coarse geometric structure seems to be particularly interesting is big mapping class groups, that is, uncountable mapping class groups. In this setting, work of K. Mann and K. Rafi \cite{rafi} completely determines for which surfaces the mapping class group is locally bounded and, under additional assumptions, when it is monogenic. In the latter case, it however remains a significant task to provide an intelligible description of the quasi-isometric type of the mapping class group, though some progress has been made \cite{cohen}.

\begin{problem}
Provide explicit descriptions of the quasi-isometry types of big mapping class groups and determine the non-trivial geometric properties that they may have? 
\end{problem}

By a non-trivial geometric property, one might mean any property of quasimetric spaces that is inherited by subspaces but fails for  universal separable metric spaces such as $C([0,1])$.

In the light of the isomorphism ${\sf MCG}_+(M) \iso {\sf Aut}(C(M))$, Theorem \ref{action}  and the proof of Theorem \ref{existence}, we know that when ${\sf MCG}_+(M)$ is monogenic, one can find an invariant subset $V\subseteq\big( {\sf Vert}\; C(M)\big)^n$ for some $n$ and an invariant connected edge relation $E\subseteq V\times V$ so that, for every $v\in V$, the evaluation map
$$
g\in {\sf MCG}_+(M) \mapsto gv\in V
$$
is a quasi-isometry. The paper \cite{cohen} investigates when, in fact, one can take $n=1$.

\begin{problem}
Describe the quasi-isometry type of big mapping class groups in terms of invariant graph structures on tuples of isotopy classes of essential simple closed curves on $M$.
\end{problem}

On the pure side of the coarse geometry theory of Polish groups, numerous interesting and fundamental questions remain open. A large number of these are collected at the end of \cite{coarsebook}. Alternatively, regarding local Lipschitz structure, the main problem is to characterise the Polish groups admitting a minimal metric.

\begin{problem}
Suppose $G$ is a Polish group with a minimal metric. Is $G$ isomorphic to a closed subgroup of a Banach--Lie group?
\end{problem}

This problem is related to Hilbert's fifth problem in infinite dimensions and would provide a significant clarification of how minimal metrics can come about.


\end{document}